\documentclass[reqno]{amsart}
\usepackage{amsmath,amssymb,amsthm,mathtools,mathrsfs}
\usepackage{microtype,booktabs,enumitem,needspace}
\usepackage{hyperref}
\hypersetup{colorlinks=true,linkcolor=black,citecolor=black,urlcolor=black,
	pdftitle={Subvariety intervals and limit varieties of zero extensions of mirror-word semirings},
	pdfsubject={Continuum intervals, axiomatic rank and limit subvarieties},
	pdfauthor={Ruijing Wang and Mengya Yue}}
\setlist[enumerate]{label=\textup{(\roman*)},leftmargin=2.5em,itemsep=2pt}
\newtheorem{theorem}{Theorem}[section]
\newtheorem{proposition}[theorem]{Proposition}
\newtheorem{lemma}[theorem]{Lemma}
\newtheorem{corollary}[theorem]{Corollary}
\theoremstyle{definition}

\theoremstyle{remark}
\newtheorem{remark}[theorem]{Remark}
\newcommand{\Var}{\mathsf V}
\newcommand{\Id}{\operatorname{Id}}
\newcommand{\Fac}{\operatorname{Fac}^{+}}
\newcommand{\mword}{\boldsymbol m}
\newcommand{\W}{\mathcal W}
\newcommand{\V}{\mathcal V}
\newcommand{\PX}{\mathcal P_X}
\title[Subvariety intervals of mirror-word semirings]{Subvariety intervals and limit varieties\\of zero extensions of mirror-word semirings}
\author{Ruijing Wang}
\author{Mengya Yue}
\address{School of Mathematics, Northwest University, Xi'an, 710127, Shaanxi, P.~R.~China}
\email{202521377@stumail.nwu.edu.cn}
\address{School of Mathematics, Northwest University, Xi'an, 710127, Shaanxi, P.~R.~China}
\email{myayue@yeah.net}
\date{}
\begin{document}
	\begin{abstract}
		For the mirror word $\mword_n=a_1\cdots a_na_n\cdots a_1$, let
		$\V_n=\Var(S(\mword_n))$ and $\W_n=\Var(S(\mword_n)^0)$.
		We show that $[\W_n,\W_{n+1}]$ ($n\ge1$) and
		$[\V_n,\W_n]$ ($n\ge2$) have cardinality $2^{\aleph_0}$.
		Every variety in $[\W_2,\bigvee_{n\ge2}\W_n]$ has infinite
		axiomatic rank. The unique limit subvariety of $\W_1$ is
		$\Var(SR_6)$. For $n\ge2$ we construct a distinct directed-path
		limit variety contained in $\W_n$ but not in $\V_n$.
	\end{abstract}
	\maketitle

\section{Introduction}
A \emph{variety} is a class of algebras closed under taking
subalgebras, homomorphic images, and arbitrary direct products.
By Birkhoff's theorem, varieties are precisely the classes
defined by sets of identities~\cite{BurrisSankappanavar81}. Such a set is called an
\emph{equational basis}. A variety is \emph{finitely based} if it
admits a finite equational basis, and \emph{nonfinitely based}
otherwise. An algebra has the corresponding property when the
variety it generates does. The \emph{finite basis problem} asks
which algebras and varieties admit finite equational bases.
A variety is \emph{hereditarily finitely based} if all its
subvarieties are finitely based; it is a \emph{limit variety} if
it is nonfinitely based but all its proper subvarieties are
finitely based. Thus limit varieties mark the boundary between
finite and nonfinite basability in subvariety lattices.

An \emph{additively idempotent semiring}, or \emph{ai-semiring},
is an algebra $(S,+,\cdot)$ in which $(S,+)$ is a commutative
idempotent semigroup, $(S,\cdot)$ is a semigroup, and the
distributive identities
\[
 x(y+z)\approx xy+xz,\qquad (x+y)z\approx xz+yz
\]
hold. Its addition induces the natural order: for $a,b\in S$, the relation $a\le b$ holds if and only if
$a+b=b$.
This is a partial order compatible with multiplication, and
$a+b$ is the join of $a$ and $b$. In particular, ai-semirings
may be viewed as semilattice-ordered semigroups. Throughout
this paper we work in the binary signature $(+,\cdot)$,
without named constants. The notation $\mathbf{p}\preceq \mathbf{u}$ for an
absorption identity between \emph{terms} will be introduced
separately in Section~\ref{sec:prelim}; it is distinct from
the order $\le$ on elements of a semiring.

The general theory of semirings and its applications is surveyed
in~\cite{Golan92,DrosteKuichVogler09}. Early work on finite
ai-semirings includes the nonfinitely based example of
Dolinka~\cite{Dolinka07}, the varieties of order-two
ai-semirings classified by Shao--Ren~\cite{ShaoRen15},
and the commutative Burnside-type varieties considered by
Ren--Zhao~\cite{RenZhao16}. Systematic classifications of
four-element ai-semirings appear in
\cite{RenLiuZengChen25,YueRenZengShao25,RenLiuYueChen26}.
Further finite- and nonfinite-basis phenomena in specific
inverse and Brandt semirings are exhibited in
\cite{Volkov21,Shaprynskii24}.

The finite basis problem for ai-semirings has also been studied in
particular for flat semirings and for semirings obtained by
adjoining a zero. Jackson--Ren--Zhao~\cite{JRZ22} developed
flat-semiring methods for nonfinite basability, and
Wu--Ren--Zhao~\cite{WuRenZhao24} proved that $S_7^0$ is
nonfinitely based. Gao--Jackson--Ren--Zhao
\cite{GaoJacksonRenZhao25} further analyzed ai-semirings
related to $S_7$. Gao--Ren--Zhao~\cite{GaoRenZhaoFlat26}
subsequently obtained general finite- and nonfinite-basis
results for flat word semirings $S(W)$, while their work on
zero adjunction~\cite{GaoRenZhao26} established continuum-sized
intervals between varieties generated by $S$ and $S^0$.
Related examples in which adjoining a zero changes the finite
basis behavior occur in~\cite{YueRenGao26}.
For the \emph{flat semirings themselves}, rather than their
zero extensions, Wang--Wang~\cite{WangWang26} established
finite identity bases for arbitrary sets of linear words.
Gao--Zhou~\cite{GaoZhou26} obtained further finite bases and
continuum-sized subvariety intervals using interval semirings
of posets. These linear-word results do not, by themselves,
settle the corresponding problems for zero extensions of
words with repeated letters.

The study of limit varieties provides another perspective on
the finite basis problem. Ren--Jackson--Zhao--Lei
\cite{RenJacksonZhaoLei23} constructed limit varieties from
flat extensions of finite groups. Lyu--Ren--Yue
\cite{LyuRenYue26} proved that the six-element semiring
$SR_6$ generates a limit variety, and Shao--Ren--Gao
\cite{ShaoRenGao26} established the corresponding property
for a different six-element semiring $TR_6$. These examples
motivate the search for limit subvarieties inside varieties
generated by zero extensions of flat word semirings.

We consider the \emph{mirror words}
\[
 \mword_n=a_1\cdots a_na_n\cdots a_1\qquad(n\ge1),
\]
where $a_1,a_2,\ldots$ are distinct letters. Let
$A_n=S(\mword_n)$ be the flat semiring of their nonempty
contiguous factors, with absorbing top element $\infty$,
and let $\widehat A_n=A_n^0$ be obtained by adjoining a new
additive identity and multiplicative zero $0\ne\infty$.
We write
\[
 \V_n=\Var(A_n),\qquad
 \W_n=\Var(\widehat A_n),\qquad
 \W_\infty=\bigvee_{n\ge2}\W_n.
\]
Since the mirror words contain repeated letters, the
associated zero extensions have features not covered by the
linear-word constructions.

Shao--Gao~\cite[Theorem~1.1]{ShaoGao26} determined the
finite basis property of $S(W)^0$ for every nonempty set of
nonempty words $W$. In particular, their theorem implies that
$\widehat A_n$ is finitely based precisely when $n=1$.
Their Corollary~4.8 also yields infinite axiomatic rank for
each $\W_n$ with $n\ge2$. More substantially, Shao--Gao
\cite[Theorems~1.3 and~14.2]{ShaoGao26} proved that
\emph{every} variety between a common limit subvariety
$\Var(B_{15})$ and the join of their linear-word zero-extension
varieties has infinite axiomatic rank. Thus interval-wide
infinite axiomatic rank is already known for zero extensions
of flat word semirings. A related uniform obstruction for
power semirings of finite groups appears in
\cite{ShaoGaoPower26}. Our result concerns a \emph{different}
interval. Indeed, $x^2y\approx xy^2$ holds in every
linear-word zero extension but fails in $\widehat A_2$
(take $x=b$, $y=a$ in $S(abba)^0$), so $\W_2$ is not
contained in their linear-word upper variety. We prove that
every variety in $[\W_2,\W_\infty]$ has infinite
axiomatic rank, using finite local quotients and a six-element
classification of evaluation kernels rather than the
linear-word argument of Shao--Gao.

We also investigate intervals of subvarieties. For every
$n\ge1$, we construct an order embedding of the power-set
order into $[\W_n,\W_{n+1}]$ using finite context quotients
and graph-separating identities. For $n\ge2$, we apply the
cyclic semirings and identities of Shao--Gao
\cite[Sections~10--11]{ShaoGao26}, after checking the necessary
endpoint conditions, to obtain a continuum-sized interval
$[\V_n,\W_n]$. These methods also imply that the chain
$\W_1\subsetneq\W_2\subsetneq\cdots$ is strict and that
$\W_\infty$ is locally finite.

Finally, we study limit subvarieties. We show that
$\Var(SR_6)$ is the unique limit subvariety of $\W_1$.
For $n\ge2$, we construct from an infinite directed path a
new limit variety $\mathcal U=\Var(T_\infty)$ such that
$\mathcal U\subseteq\W_n$ but $\mathcal U\nsubseteq\V_n$.
Moreover,
\[
 \mathcal U\cap\Var(B_{15})=\Var(S(abc)),
\]
and $\mathcal U$ is incomparable with both limit varieties
constructed by Shao--Gao
\cite[Theorems~1.2 and~22.1]{ShaoGao26}. Together with the
known acyclic-graph limit variety, these results yield at
least five pairwise incomparable limit subvarieties of
$\W_n$ for each $n\ge2$.

\begin{theorem}[Shao--Gao]\label{thm:main}
For every integer $n\ge1$, the semiring $\widehat A_n$ is
finitely based if and only if $n=1$.
\end{theorem}

\Needspace{7\baselineskip}
\begin{theorem}\label{thm:overview}
The inclusions $\W_1\subsetneq\W_2\subsetneq\cdots$ are strict.
Each interval $[\W_n,\W_{n+1}]$, $n\ge1$, contains an
order-embedded copy of $(\mathcal P(\mathbb N),\subseteq)$ and has
cardinality $2^{\aleph_0}$. The same cardinality and power-set
embedding hold for $[\V_n,\W_n]$ when $n\ge2$, by applying
the cyclic construction of Shao--Gao; also
$|[\V_n,\W_\infty]|=2^{\aleph_0}$.
Every $\mathcal K$ with $\W_2\subseteq\mathcal K\subseteq\W_\infty$
has infinite axiomatic rank.
The variety $\Var(SR_6)$ is the unique limit subvariety of $\W_1$.
For every $n\ge2$, a new limit variety
$\mathcal U=\Var(T_\infty)$ lies in $\W_n$ but not in $\V_n$;
it is incomparable with the two limit varieties of Shao--Gao.
In particular, $\W_n$ contains at least five pairwise
incomparable limit subvarieties.
\end{theorem}

Section~\ref{sec:prelim} gives the notation and basic properties.
Section~\ref{sec:nonfinite} treats the uniform axiomatic-rank
obstruction, and Section~\ref{sec:variety} establishes the
continuum intervals and local finiteness.
Sections~\ref{sec:unique-limit} and~\ref{sec:path-limit-section}
prove the limit-subvariety results.

\section{Preliminaries and mirror-word zero extensions}\label{sec:prelim}

For a nonempty word $\mathbf{w}$, let $\Fac(\mathbf{w})$ be the set of its
nonempty contiguous factors. As in the standard construction
\cite[Lemma~2.5]{JRZ22}, the \emph{flat word semiring} $S(\mathbf{w})$
has underlying set $\Fac(\mathbf{w})\cup\{\infty\}$ and operations
\[
 \mathbf{u}+\mathbf{v}=\begin{cases}\mathbf{u},&\mathbf{u}=\mathbf{v},\\\infty,&\mathbf{u}\ne \mathbf{v},\end{cases}
 \qquad
 \mathbf{u}\cdot \mathbf{v}=\begin{cases}\mathbf{u}\mathbf{v},&\mathbf{u}\mathbf{v}\in\Fac(\mathbf{w}),\\\infty,&\text{otherwise},\end{cases}
\]
Here $\infty$ absorbs both operations. The factors are words in
the free semigroup, so no commutativity is imposed; in particular,
$S(\mathbf{w})$ is distinguished from the commutative word semiring
$S_c(\mathbf{w})$. For an ai-semiring $S$, its \emph{zero extension} $S^0$
is obtained by adjoining a new element $0$ such that
\[
 0+s=s+0=s,\qquad 0s=s0=0\qquad(s\in S^0).
\]
The new element $0$ is distinct from every element of $S$,
including any multiplicative zero already in $S$.

\begin{lemma}\label{lem:basic}
For every $n\ge1$, the mirror-word zero extension satisfies
$|\widehat A_n|=2n^2+2$. Moreover, every $x\in A_n$ satisfies
$x^3=\infty$, whereas $0^3=0$ in $\widehat A_n$.
The two-element distributive lattice semiring
$D_2=(\{0,1\},\vee,\wedge)$ is a homomorphic image of
$\widehat A_n$.
\end{lemma}
\begin{proof}
The adjacent factors of length two in $\mword_n$ are pairwise
distinct. Consequently, every factor of length at least two
occurs exactly once, while the factors of length one are the $n$
distinct letters. Thus the number of nonempty factors is
$\binom{2n}{2}+n=2n^2$. Adjoining $\infty$ and $0$ gives the
asserted cardinality.
Since each letter occurs exactly twice in $\mword_n$, the cube
of any nonempty factor cannot be a factor; the same identity
$x^3=\infty$ holds when $x=\infty$.
Finally, sending the adjoined $0$ to $0$ and every element of
$A_n$ to $1$ gives a surjective semiring homomorphism onto $D_2$.
The same homomorphism exists for the zero extension of any
nonempty ai-semiring.
\end{proof}

Let $X$ be a countably infinite set of variables and $X^+$ the
free semigroup on $X$. By distributivity and additive
idempotence, an ai-semiring term may be identified with a finite
nonempty set of words in $X^+$; see
\cite[Theorem~2.5]{KurilPolak05}.
We use bold letters for formal words and terms, and ordinary italic letters for variables and elements of semirings.
We write $\mathbf{u}=\mathbf{u}_1+\cdots+\mathbf{u}_r$ for the term corresponding to
$\{\mathbf{u}_1,\ldots,\mathbf{u}_r\}$, and regard its word summands as members
of $\mathbf{u}$. For a word $\mathbf{p}$, let $c(\mathbf{p})$ denote its \emph{content}
and $\ell(\mathbf{p})$ its \emph{length}.
For a term $\mathbf{u}$, put $c(\mathbf{u})=\bigcup_{\mathbf{p}\in \mathbf{u}}c(\mathbf{p})$.
For elements $a,b$ of an ai-semiring, the \emph{natural order}
is defined by $a\le b$ if and only if $a+b=b$.
Following the identity notation of~\cite[Section~2]{ShaoRenGao26},
for ai-semiring terms $\mathbf{p},\mathbf{u}$ write $\mathbf{p}\preceq \mathbf{u}$ (or, equivalently,
$\mathbf{u}\succeq \mathbf{p}$) for the \emph{absorption identity}
$\mathbf{u}\approx \mathbf{u}+\mathbf{p}$. In particular, $S\models \mathbf{p}\preceq \mathbf{u}$
means that $\varphi(\mathbf{p})\le\varphi(\mathbf{u})$ in the natural order of
$S$ for every evaluation $\varphi$. Thus $\preceq$ denotes an
identity between formal terms, whereas $\le$ compares elements
of a semiring. Every identity between finite sums of words is
equivalent to its word-absorption identities; it therefore suffices
to consider $\mathbf{q}\preceq \mathbf{u}$ with $\mathbf{q}\in X^+$.

For a word $\mathbf{q}\in X^+$ and a term $\mathbf{u}$, following
Gao--Ren--Zhao~\cite[Proposition~2.1]{GaoRenZhao26}
and Shao--Gao~\cite[Eq.~(2)]{ShaoGao26}, put
\[
 D_\mathbf{q}(\mathbf{u})=\{\mathbf{p}\in \mathbf{u}:c(\mathbf{p})\subseteq c(\mathbf{q})\}.
\]
The set $D_\mathbf{q}(\mathbf{u})$ may be empty. In expressions involving
$\mathbf{q}\preceq D_\mathbf{q}(\mathbf{u})$, nonemptiness is explicitly required, since
an empty sum is not a term in the present binary signature.

The absorption criterion for zero adjunction gives a convenient
test for arbitrary term identities by applying it to each word
summand on both sides.
\begin{lemma}[Identity criterion for zero extensions]\label{lem:restriction}
Let $S$ be an ai-semiring, and let $\mathbf{u},\mathbf{v}$ be ai-semiring terms.
Then $S^0\models \mathbf{u}\approx \mathbf{v}$ if and only if
\[
\begin{aligned}
&D_\mathbf{p}(\mathbf{v})\ne\varnothing\ \text{and}\ S\models \mathbf{p}\preceq D_\mathbf{p}(\mathbf{v})
 &&\text{for every }\mathbf{p}\in \mathbf{u},\\
&D_\mathbf{q}(\mathbf{u})\ne\varnothing\ \text{and}\ S\models \mathbf{q}\preceq D_\mathbf{q}(\mathbf{u})
 &&\text{for every }\mathbf{q}\in \mathbf{v}.
\end{aligned}
\]
\end{lemma}
\begin{proof}
An identity $\mathbf{u}\approx \mathbf{v}$ is equivalent to the absorptions
$\mathbf{p}\preceq \mathbf{v}$ for all $\mathbf{p}\in \mathbf{u}$ and $\mathbf{q}\preceq \mathbf{u}$ for all
$\mathbf{q}\in \mathbf{v}$. Applying the zero-adjunction absorption criterion
of Wu--Ren--Zhao~\cite[Proposition~1.5]{WuRenZhao24} (see
also Gao--Ren--Zhao~\cite[Proposition~2.1]{GaoRenZhao26})
to each of these absorptions gives precisely the stated
conditions. The criterion and its proof are recalled below.
\end{proof}

The following word-absorption criterion is an established
result of Wu--Ren--Zhao~\cite[Proposition~1.5]{WuRenZhao24}
and Gao--Ren--Zhao~\cite[Proposition~2.1]{GaoRenZhao26};
see also Shao--Gao~\cite[Lemma~4.5]{ShaoGao26} for flat
word semirings. We include a short proof for completeness.
\begin{proposition}[Zero-adjunction criterion]\label{cor:absorption}
Let $S$ be an ai-semiring, $\mathbf{u}$ an ai-semiring term, and
$\mathbf{q}\in X^+$ a word. Then the zero extension $S^0$ satisfies $\mathbf{q}\preceq \mathbf{u}$ if and
only if both $D_\mathbf{q}(\mathbf{u})\ne\varnothing$ and
$S\models \mathbf{q}\preceq D_\mathbf{q}(\mathbf{u})$.
\end{proposition}
\begin{proof}
Suppose $S^0\models \mathbf{q}\preceq \mathbf{u}$. Assign the adjoined zero to every
variable outside $c(\mathbf{q})$ and arbitrary elements of $S$ to those in
$c(\mathbf{q})$. The word $\mathbf{q}$ then has its value in $S$. Hence at least
one word of $\mathbf{u}$ must survive, so $D_\mathbf{q}(\mathbf{u})\ne\varnothing$; the
surviving words also give $S\models \mathbf{q}\preceq D_\mathbf{q}(\mathbf{u})$.

Conversely, consider an arbitrary evaluation in $S^0$. If a
variable occurring in $\mathbf{q}$ is assigned $0$, then $\mathbf{q}$ evaluates
to $0$, which is the additive least element of $S^0$. Otherwise
$\mathbf{q}$ and every word of $D_\mathbf{q}(\mathbf{u})$ evaluate in $S$, and the assumed absorption identity gives
$\varphi(\mathbf{q})\le\varphi(D_\mathbf{q}(\mathbf{u}))\le\varphi(\mathbf{u})$ in the natural
order of $S^0$. Thus $S^0\models \mathbf{q}\preceq \mathbf{u}$.
\end{proof}

\section{Infinite axiomatic rank in the mirror-word hierarchy}\label{sec:nonfinite}
Shao--Gao~\cite[Theorem~1.1 and Corollary~4.8]{ShaoGao26}
already determine the finite-basis property and infinite axiomatic
rank of each individual $\widehat A_n$. We address the stronger question of infinite axiomatic rank
for every intermediate variety in $[\W_2,\W_\infty]$.
The proof constructs finite countermodels from identities of
$\widehat A_2$ and does not invoke the equational theory of
the unextended word semirings.

\begin{theorem}\label{ind:thm:main}
Every variety $\mathcal K$ of ai-semirings satisfying
$\W_2\subseteq\mathcal K\subseteq\W_\infty$ has
\emph{infinite axiomatic rank}; that is, no identity basis of
$\mathcal K$ has a uniform finite bound on the number of
variables. In particular, it is nonfinitely based.
\end{theorem}

\subsection{Finite local quotients}
Fix a nonempty finite variable set $X$. Define an equivalence on $X^+$ by keeping every word of length
at most four in a singleton class and identifying two words of
length at least five exactly when they have the same content.
This is a semigroup congruence: multiplication preserves the
length threshold, and the content of a product is the union of
the contents of its factors. Denote the resulting finite quotient
semigroup by $\overline{X^+}$. Identifying short words with
their singleton classes, we obtain
\[
\overline{X^+}=\{\mathbf{w}\in X^+:1\le \ell(\mathbf{w})\le 4\}\ \dot\cup\
\{\lambda_C:\varnothing\ne C\subseteq X\}.
\]
Here $\lambda_C$ is the class of words of length at least five with
content $C$. If $|X|=r$, then
$|\overline{X^+}|=r+r^2+r^3+r^4+2^r-1$.

Consider its \emph{power semiring}
\[
\PX:=\bigl(\mathcal P(\overline{X^+}),\cup,\cdot\bigr),\qquad
UV=\{\mathbf{u}\mathbf{v}:\mathbf{u}\in U,\ \mathbf{v}\in V\}.
\]
The empty subset is the additive identity and the
multiplicative zero. We identify each singleton subset with
its unique member.

\begin{lemma}[Evaluation extension]\label{ind:lem:evaluation}
Every assignment $\alpha:X\to\widehat A_2$ extends uniquely
to a semiring homomorphism
$\widehat\alpha:\PX\to\widehat A_2$ with
$\widehat\alpha(\varnothing)=0$. In particular,
$\widehat\alpha(\lambda_C)=0$ if some variable of $C$
has value $0$, and it equals $\infty$ otherwise.
\end{lemma}
\begin{proof}
For a word of length at least five, an evaluation gives the
new zero $0$ whenever one of its variables takes that value.
Under an all-old evaluation the word instead evaluates to
$\infty$, since five nonempty old factors cannot concatenate
to a factor of $abba$. Its value therefore depends only on
its variable content, which makes evaluation well defined on
$\overline{X^+}$.

Extend evaluation to a subset of $\overline{X^+}$ by summing
the values of its members, with the empty subset assigned $0$.
Union is preserved by construction, and distributivity gives
preservation of setwise multiplication. The resulting homomorphism
extends $\alpha$ and is unique, since the singleton classes
and the empty subset generate $\PX$.
\end{proof}

Let $\Id_{\le k}(\widehat A_2)$ be the set of identities
of $\widehat A_2$ involving at most $k$ variables. Denote by
$\Theta_{k,X}$ the semiring congruence on $\PX$ generated
by all substitution instances of these identities, allowing
arbitrary substitutions into $\PX$. Put
\[
D_{k,X}=\PX/\Theta_{k,X}.
\]
\begin{proposition}[Finite local model]\label{ind:lem:local}
The semiring $D_{k,X}$ is finite and satisfies every identity of $\widehat A_2$ involving at most $k$ variables. Moreover, if $U\mathrel{\Theta_{k,X}}V$, then $U$ and $V$ have the same value under every assignment $X\to\widehat A_2$.
\end{proposition}
\begin{proof}
Every substitution instance of an identity involving at most
$k$ variables belongs to the generating congruence. Thus the
quotient satisfies each of these identities, as claimed.

To verify that the quotient relation preserves all evaluations,
fix $\alpha:X\to\widehat A_2$ and let
$\widehat\alpha$ be its extension from
Lemma~\ref{ind:lem:evaluation}. Every generating pair of
$\Theta_{k,X}$ lies in $\ker\widehat\alpha$, because
substitution in $\PX$ followed by $\widehat\alpha$ is an
ordinary evaluation in $\widehat A_2$. Since the kernel is a
congruence, it contains all of $\Theta_{k,X}$. Thus
$\Theta_{k,X}\subseteq\ker\widehat\alpha$, proving
the second assertion. Finally, $D_{k,X}$ is finite because $\PX$ is finite.
\end{proof}

\subsection{Separating identities and their evaluations}
We may freely rename the letters of $\mword_n$ as
$t_1,\ldots,t_n$ when discussing its factors.
Fix $m\ge2$ and introduce
\begin{align*}
Z_m&=\{a,b_0,\ldots,b_m,c_0,\ldots,c_{m-1}\},& |Z_m|&=2m+2,\\
\mathbf{Q}_m&=ab_0^2a+ab_ma+\sum_{i=0}^{m-1}(ab_ic_i+ab_{i+1}c_i),\\
\mathbf{K}_m&=\left(\prod_{z\in Z_m}z\right)^5.
\end{align*}
The order in the product is fixed but arbitrary. Its image in $\overline{Z_m^+}$ is $\lambda_{Z_m}$.

\begin{lemma}[Rigidity of the central square]\label{ind:lem:center}
Fix $n\ge2$ and, within this proof and the next, write
$\mword_n=t_1\cdots t_nt_n\cdots t_1$.
If nonempty words $\mathbf{s},\mathbf{t}$ satisfy $\mathbf{s}\mathbf{t}^2\mathbf{s}\in\Fac(\mword_n)$, then
\[
\mathbf{t}=t_n,\qquad \mathbf{s}=t_{n-1},\qquad
\mathbf{s}\mathbf{t}^2\mathbf{s}=t_{n-1}t_nt_nt_{n-1}.
\]
\end{lemma}
\begin{proof}
Number the positions of $\mword_n$ by $1,\ldots,2n$; the two
occurrences of $t_j$ are at $j$ and $2n+1-j$.  If the first
copy of $\mathbf{t}$ begins with $t_j$, the second begins at the other
occurrence, so
\[
\ell(\mathbf{t})=(2n+1-j)-j=2n+1-2j.
\]
For $j<n$ this length is at least $3$.  The second letter of
the first copy is then $t_{j+1}$ (at position $j+1$), whereas
the corresponding position in the second copy is $2n+2-j$,
which does not carry $t_{j+1}$.  Thus $j=n$ and $\mathbf{t}=t_n$.

The two copies of $\mathbf{t}$ occupy positions $n,n+1$.
Writing $h=\ell(\mathbf{s})\ge1$, the first letters of the two
copies of $\mathbf{s}$ occur at $n-h$ and $n+2$, respectively.  Hence
\[
t_{n-h}=t_{n-1},\qquad h=1,\qquad \mathbf{s}=t_{n-1},
\]
and the asserted central factor follows.
\end{proof}

\begin{lemma}[Global validity]\label{ind:lem:global}
For every $m\ge2$, the identity
\[
\delta_m:\qquad \mathbf{Q}_m\approx \mathbf{Q}_m+\mathbf{K}_m
\]
holds throughout $\W_\infty$.
\end{lemma}
\begin{proof}
Fix $n\ge2$ and an evaluation in $\widehat A_n$.  If any
variable takes the new zero $0$, then $\mathbf{K}_m=0$, and the
absorption is automatic.  Suppose all variables take old
values.  Every such value is either $\infty$ or a nonempty
factor of $\mword_n$.  Thus
\[
\mathbf{K}_m=\infty:
\]
this holds if some factor is $\infty$, and otherwise because
$\mathbf{K}_m$ repeats every variable five times whereas each letter of
$\mword_n$ occurs only twice.

Assume $\mathbf{Q}_m\ne\infty$.  Flat addition forces all its
summands to be the same old non-top factor.  Applying
Lemma~\ref{ind:lem:center} to $ab_0^2a$, and then cancelling
the external letters in $ab_ma$, gives
\[
\begin{gathered}
a=t_{n-1},\qquad b_0=t_n,\qquad b_m=t_n^2.
\end{gathered}
\]
For $0\le i<m$ the summands $ab_ic_i$ and $ab_{i+1}c_i$
also have the same non-top value.  Left and right cancellation
in the free semigroup gives
\[
 b_i=b_{i+1}\quad(0\le i<m).
\]
Hence $b_m=b_0=t_n$, contrary to $b_m=t_n^2$ since $t_n$ is a nonempty word.
Therefore $\mathbf{Q}_m=\infty$ under every all-old evaluation, and
$\mathbf{Q}_m+\mathbf{K}_m=\mathbf{Q}_m$ in both cases.  Since $n$ was arbitrary,
$\delta_m$ holds in $\W_\infty$.
\end{proof}

\paragraph{Deletion assignments and absorption.}
Arrange the variables other than $a$ in a path
\[
P_m:\quad b_0,c_0,b_1,c_1,\ldots,c_{m-1},b_m.
\]
For every path vertex $x$, define an assignment
$\theta_x:Z_m\to\widehat A_2$ by
$\theta_x(a)=a$, $\theta_x(x)=0$, and the following table:
\[
\begin{array}{c|cc}
 & yx & xy\\\hline
 y=b_i&b&bb\\
 y=c_i&ba&a
\end{array}.
\]
Here $yx$ and $xy$ label only the relative positions of $y$
and $x$ along $P_m$, not necessarily adjacent letters in a word.
Let
\[
\begin{aligned}
\mathbf{s}_m&=c_{m-1}b_ma,\qquad \mathbf{t}_m=ab_mc_{m-1},\\
\mathscr S_m&=\{ab_0^2a,ab_ma,\mathbf{s}_m\}\\
&\quad\cup\{ab_ic_i,ab_{i+1}c_i:0\le i<m\}.
\end{aligned}
\]
For subsequent use, define the sum
\[
\mathbf{R}_m=ab_0^2a+ab_ma+
\sum_{i=0}^{m-2}(ab_ic_i+ab_{i+1}c_i)+ab_{m-1}c_{m-1}.
\]
Thus $\mathbf{Q}_m=\mathbf{R}_m+\mathbf{t}_m$; the sum $\mathbf{R}_m$ contains precisely
those summands of $\mathbf{Q}_m$ other than $\mathbf{t}_m$.

\begin{lemma}[Deletion assignment]\label{ind:lem:deletion}
For each $x\in P_m$ and each $\mathbf{v}\in\mathscr S_m$,
\[
\mathbf{v}^{\theta_x}=\begin{cases}0&x\in c(\mathbf{v}),\\abba&x\notin c(\mathbf{v}).\end{cases}
\qquad\text{and}\qquad \mathbf{Q}_m^{\theta_x}=abba.
\]
\end{lemma}
\begin{proof}
For each relational monomial $ab_ic_j$, the vertices $b_i,c_j$
are adjacent in $P_m$. If neither vertex is $x$, they lie on the
same side of $x$. In either case, substituting their prescribed
images into $ab_ic_j$ gives $abba$. When defined away from the deleted vertex, the left and
right anchors evaluate respectively as $ab^2a=abba$ and
$a(bb)a=abba$. The additional word $\mathbf{s}_m$ also evaluates to $abba$ whenever neither of its adjacent path vertices is deleted: they then lie to the right of $x$. Any word containing $x$ evaluates to $0$. At least one of the two anchor monomials avoids $x$, so the sum $\mathbf{Q}_m$ evaluates to $abba$.
\end{proof}

\begin{lemma}[A three-variable exchange identity]\label{ind:lem:exchange}
The following identity holds in $\widehat A_2$:
\[
xyx+xyz\approx xyx+zyx.
\]
Consequently,
\[
\mathbf{R}_m+\mathbf{t}_m\approx \mathbf{R}_m+\mathbf{t}_m+\mathbf{s}_m\approx \mathbf{R}_m+\mathbf{s}_m
\]
holds in $\widehat A_2$.
\end{lemma}
\begin{proof}
If any variable is assigned $0$, direct inspection shows that both sides agree. For all-old assignments, if $xyx=\infty$, both sides are $\infty$. Otherwise, the only non-top value of a product $xyx$ in $S(abba)$ is $abba$, necessarily with $x=a$ and $y=bb$. Then $xyz$ and $zyx$ are simultaneously $abba$ when $z=a$, and otherwise both are $\infty$. This establishes the identity. Substitute $x=a$, $y=b_m$, $z=c_{m-1}$ and add the remaining monomials of $\mathbf{R}_m$. Adding the two equivalent forms by idempotent addition yields the intermediate sum with both $\mathbf{t}_m$ and $\mathbf{s}_m$.
\end{proof}

\begin{lemma}\label{ind:lem:short}
If $m\ge2$, $\mathbf{w}\in Z_m^+$, and $\ell(\mathbf{w})\le4$, then
the identity $\mathbf{Q}_m\approx \mathbf{Q}_m+\mathbf{w}$ holds in $\widehat A_2$
if and only if $\mathbf{w}\in\mathscr S_m$.
\end{lemma}
\begin{proof}
Every $\mathbf{w}\in\mathscr S_m$ is absorbed by $\mathbf{Q}_m$ by
idempotence and Lemma~\ref{ind:lem:exchange}.

Conversely, suppose $\widehat A_2\models \mathbf{Q}_m\approx \mathbf{Q}_m+\mathbf{w}$.
We must have $a\in c(\mathbf{w})$: assigning $a=0$ and the remaining
variables nonzero old factors gives $\mathbf{Q}_m=0$ but $\mathbf{w}\ne0$ if
$a\notin c(\mathbf{w})$.  For any omitted path vertex
$x\in P_m\setminus c(\mathbf{w})$, Lemma~\ref{ind:lem:deletion}
gives $\mathbf{Q}_m^{\theta_x}=abba$. Since $\mathbf{w}$ is absorbed by
$\mathbf{Q}_m$, flat addition then gives $\mathbf{w}^{\theta_x}=abba$.
Such a vertex exists, since $|P_m|=2m+1\ge5$ and $\mathbf{w}$
contains $a$ and at most three distinct path variables.
All nonzero $\theta_x$-images lie in $\{a,b,bb,ba\}$.
The only decompositions of $abba$ into three or four factors
from this set are
\[
abba=a\cdot b\cdot b\cdot a
     =a\cdot b\cdot ba
     =a\cdot bb\cdot a.
\]
Consequently $\ell(\mathbf{w})\in\{3,4\}$.

\smallskip\noindent\emph{Case $\ell(\mathbf{w})=4$.}
The middle letters are $b$-variables; the endpoints are
$a$ or $c$-variables, with at least one endpoint equal to $a$.
If the other endpoint is $c_h$ and $b_m\notin c(\mathbf{w})$, deletion
of $b_m$ makes $c_h$ evaluate to $ba$, not the required
endpoint $a$.  If $b_m\in c(\mathbf{w})$, it occurs in a middle
position; an omitted path vertex preceding $b_m$ makes its
image $bb$, not the required one-letter factor $b$.
Thus both endpoints are $a$ and $\mathbf{w}=ab_ib_ja$.
Deleting $c_0$ forces both middle variables to be $b_0$:
\[
\mathbf{w}=ab_0^2a.
\]

\smallskip\noindent\emph{Case $\ell(\mathbf{w})=3$.}
The middle position is a $b_i$, the endpoints belong to
$\{a\}\cup\{c_h:0\le h<m\}$, and at least one is $a$.
There are three possibilities:
\[
\mathbf{w}\in\{c_hb_ia,\ ab_ia,\ ab_ic_j\}.
\]
For $c_hb_ia$, deletion of $b_m$ excludes $i<m$, and
deletion of $c_{m-1}$ excludes $h<m-1$; hence
$\mathbf{w}=c_{m-1}b_ma=\mathbf{s}_m$.  For $ab_ia$, deletion of $b_m$
excludes $i<m$, giving $\mathbf{w}=ab_ma$.  Finally, for $ab_ic_j$,
the vertices $b_i,c_j$ must lie on the same side of every
omitted path vertex; otherwise their prescribed images cannot
multiply with $a$ to give $abba$.  Hence they are adjacent
in $P_m$, and
\[
\mathbf{w}\in\{ab_ic_i,\ ab_{i+1}c_i:0\le i<m\}.
\]
These are precisely the remaining members of $\mathscr S_m$.
\end{proof}

\begin{lemma}\label{ind:lem:long}
For every $\mathbf{w}\in Z_m^+$ with $\ell(\mathbf{w})\ge5$, the identity $\mathbf{Q}_m\approx \mathbf{Q}_m+\mathbf{w}$ holds in $\widehat A_2$
if and only if $c(\mathbf{w})=Z_m$.
\end{lemma}
\begin{proof}
If the content is all of $Z_m$, any assignment with a new zero sends $\mathbf{w}$ to $0$, and any all-old assignment sends both $\mathbf{Q}_m$ (Lemma~\ref{ind:lem:global}) and $\mathbf{w}$ to $\infty$. Conversely, if $a\notin c(\mathbf{w})$, assign $a=0$ and all other variables nonzero old values: then $\mathbf{Q}_m=0$ and $\mathbf{w}=\infty$. If a path vertex $x$ is omitted from the support, use $\theta_x$: it sends $\mathbf{Q}_m$ to $abba$ but sends the length-at-least-five word $\mathbf{w}$, involving only nonzero old values, to $\infty$. Thus the identity fails unless the content is all of $Z_m$.
\end{proof}

\subsection{The six-element kernel class}
Hereafter fix $X=Z_m$. As in the singleton convention for
$\PX$, the notation $\lambda_{Z_m}$ also denotes
$\{\lambda_{Z_m}\}\in\PX$ when used as an additive
summand. For $U,V\in\PX$, write $U\sim_m V$ if and only if
\[
 \widehat\alpha(U)=\widehat\alpha(V)
 \quad\text{for every }\alpha:Z_m\to\widehat A_2.
\]
This is the kernel congruence of the joint evaluation map.
Since $\Theta_{k,Z_m}\subseteq\ker\widehat\alpha$ for each $\alpha$,
we have $\Theta_{k,Z_m}\subseteq\sim_m$.

\begin{lemma}[Necessary short monomials]\label{ind:lem:necessary}
If $U\sim_m \mathbf{Q}_m$, then the short-word part of $U$ contains every monomial of $\mathbf{R}_m$ and at least one of $\mathbf{t}_m,\mathbf{s}_m$.
\end{lemma}
\begin{proof}
First, $U\sim_m \mathbf{Q}_m$ implies
$\widehat A_2\models \mathbf{Q}_m+U\approx \mathbf{Q}_m$.
Thus every member of $U$ is absorbed by $\mathbf{Q}_m$.  By
Lemmas~\ref{ind:lem:short} and \ref{ind:lem:long},
\[
U\subseteq\mathscr S_m\cup\{\lambda_{Z_m}\}.
\]
We now force the necessary short summands.  In each test
below, all unspecified variables take the new zero; hence
$\lambda_{Z_m}$ also has value $0$.
\begin{enumerate}[label=\textup{(\roman*)},leftmargin=2.2em]
\item Setting $a=a$, $b_0=b$ isolates $ab_0^2a=abba$.
Setting $a=a$, $b_m=bb$ isolates $ab_ma=abba$.
\item Setting $a=a$, $b_0=b$, $c_0=a$ gives
\[
(ab_0^2a)^\alpha=abba,\qquad
(ab_0c_0)^\alpha=\infty,
\]
while all other allowed monomials have value $0$; here
$aba$ is not a factor of $abba$.  This forces $ab_0c_0$.
\item For $1\le i<m$, setting
$a=a$, $b_i=b$, $c_i=ba$ isolates
$(ab_ic_i)^\alpha=abba$.
For $0\le i<m-1$, setting
$a=a$, $b_{i+1}=b$, $c_i=ba$ isolates
$(ab_{i+1}c_i)^\alpha=abba$.
\item Finally, setting $a=a$, $b_m=bb$, $c_{m-1}=b$ gives
\[
(ab_ma)^\alpha=abba,\qquad
(\mathbf{t}_m)^\alpha=(\mathbf{s}_m)^\alpha=\infty,
\]
with every other allowed monomial taking value $0$.
Therefore at least one of $\mathbf{t}_m,\mathbf{s}_m$ occurs in $U$.
\end{enumerate}
In each of (i)--(iii), omitting the indicated monomial would
make $U$ and $\mathbf{Q}_m$ evaluate differently; (iv) rules out
omitting both $\mathbf{t}_m,\mathbf{s}_m$.  This proves the claim.
\end{proof}

\begin{theorem}[The six-element kernel class]\label{ind:thm:fiber}
Put
\[
F_m=\{\mathbf{R}_m+\mathbf{t}_m,\ \mathbf{R}_m+\mathbf{s}_m,\ \mathbf{R}_m+\mathbf{t}_m+\mathbf{s}_m\},
\qquad F_m+\lambda_{Z_m}:=\{U+\lambda_{Z_m}:U\in F_m\}.
\]
Then, for every $m\ge2$,
\[
[\mathbf{Q}_m]_{\sim_m}=F_m\ \dot\cup\ (F_m+\lambda_{Z_m}).
\]
In particular, this kernel class has exactly six elements.
\end{theorem}
\begin{proof}
By Lemmas~\ref{ind:lem:short} and \ref{ind:lem:long}, the only possible members of an element semantically equivalent to $\mathbf{Q}_m$ are the short monomials of $\mathscr S_m$ and $\lambda_{Z_m}$. By Lemma~\ref{ind:lem:necessary}, the short part contains all monomials in $\mathbf{R}_m$ and at least one of $\mathbf{t}_m,\mathbf{s}_m$. This yields only the six stated candidates. Lemma~\ref{ind:lem:exchange} makes the three short candidates equivalent, while Lemma~\ref{ind:lem:global} makes adjoining $\lambda_{Z_m}$ semantically inert. Hence the kernel class contains exactly these six elements.
\end{proof}

\begin{lemma}\label{ind:lem:crossing}
If $[\mathbf{Q}_m]_{\Theta_{k,Z_m}}=[\mathbf{Q}_m+\mathbf{K}_m]_{\Theta_{k,Z_m}}$, then some elementary congruence step generated by a valid identity in at most $k$ variables, after substitution and application of a unary polynomial context, connects one member of $F_m$ with one of $(F_m+\lambda_{Z_m})$.
\end{lemma}
\begin{proof}
A congruence generated by a family of pairs consists of
the pairs connected by finite chains of elementary replacements,
allowing substitutions, unary polynomial contexts, and reversed
steps. Apply this description to the presumed equality. Proposition~\ref{ind:lem:local} ensures that each step preserves every evaluation in $\widehat A_2$, so every intermediate element is one of the six listed in Theorem~\ref{ind:thm:fiber}. The initial element is $\mathbf{Q}_m\in F_m$; the final one is
$\mathbf{Q}_m+\mathbf{K}_m\in(F_m+\lambda_{Z_m})$. Some adjacent pair must therefore
lie in different parts of the kernel class. We use only that a finite chain connecting the two parts
must contain a crossing; the parts need not be congruence classes.
In particular, no bounded-variable derivation between distinct
short representatives is assumed. The argument also applies
when $k=1$ or $2$.
\end{proof}

\subsection{The uniform obstruction}
\begin{proposition}[Absorption normal form]\label{ind:lem:normal}
Over the axioms of additively idempotent semirings, an identity
$\mathbf{p}=\mathbf{u}_1+\cdots+\mathbf{u}_r\approx \mathbf{q}=\mathbf{v}_1+\cdots+\mathbf{v}_s$ is equivalent to the conjunction
\[
\mathbf{p}\approx \mathbf{p}+\mathbf{v}_j\quad(1\le j\le s),\qquad
\mathbf{q}\approx \mathbf{q}+\mathbf{u}_i\quad(1\le i\le r).
\]
Consequently, a replacement by an identity valid in $\widehat A_2$
can be factored into finitely many replacements by valid
single-monomial absorption identities, without introducing
new source variables.
\end{proposition}
\begin{proof}
If $\mathbf{p}\approx \mathbf{q}$, then $\mathbf{p}\approx \mathbf{p}+\mathbf{q}$. Since
$\mathbf{p}+\mathbf{v}_j$ is an additive subterm of $\mathbf{p}+\mathbf{q}$, each $\mathbf{v}_j$
is absorbed by $\mathbf{p}$; interchanging the two sides gives
the second family. Conversely, summing the first family yields $\mathbf{p}\approx \mathbf{p}+\mathbf{q}$, and summing the second yields $\mathbf{q}\approx \mathbf{p}+\mathbf{q}$, hence $\mathbf{p}\approx \mathbf{q}$. The same deductions, under any substitution and any fixed context, supply a finite chain from one side to the other.
\end{proof}

\begin{lemma}[Linearization of polynomial contexts]\label{ind:lem:contexts}
By distributivity, a polynomial context containing its
placeholder $z$ exactly once has the schematic form
\[
\mathbf{E}+\mathbf{M}z\mathbf{N},
\]
where $\mathbf{E}$ is a term independent of $z$ (possibly absent), and
$\mathbf{M},\mathbf{N}$ are products independent of $z$ (either factor may be
absent). These are local placeholders for the context, reused
in the proof of Theorem~\ref{ind:thm:bound}, rather than
additional operations or distinguished elements of the
semiring. A context containing several occurrences of $z$
can be handled by replacing the occurrences one at a time.
Omitted summands and factors are purely notational; they
do not introduce named additive or multiplicative constants.
\end{lemma}
\begin{proof}
For a context with one occurrence of $z$, induct on its
expression tree.  The base $z$ has the stated form with
$\mathbf{E},\mathbf{M},\mathbf{N}$ formally omitted.  Adding a $z$-free term enlarges
$\mathbf{E}$; multiplying by a $z$-free term on either side, followed
by distributivity, incorporates that term into the
corresponding multiplier of the unique $z$.  Thus the context
has form $\mathbf{E}+\mathbf{M}z\mathbf{N}$, with the omissions interpreted as in the
statement.

For a context with several occurrences, distinguish them as
$z_1,\ldots,z_r$ and write $C(x_1,\ldots,x_r)$ for
the same expression with these occurrences separately named.
The successive terms in the replacement sequence have the forms
\[
\begin{gathered}
 C(\mathbf{p}^\beta,\ldots,\mathbf{p}^\beta),\\
 C(\mathbf{q}^\beta,\mathbf{p}^\beta,\ldots,\mathbf{p}^\beta),\\
 \ldots,\\
 C(\mathbf{q}^\beta,\ldots,\mathbf{q}^\beta).
\end{gathered}
\]
At each step, one occurrence is replaced using the same
substituted identity, with the other occurrences fixed.
Every step preserves evaluation in $\widehat A_2$; hence
all intermediate terms lie in $[\mathbf{Q}_m]_{\sim_m}$ whenever the
initial one does. The steps can also be performed in reverse order.
\end{proof}

\begin{lemma}[Common old-factor realization on a proper support]\label{ind:lem:realization}
If $S\subsetneq Z_m$ and $a\in S$, there is an assignment $\alpha:Z_m\to\widehat A_2$ such that every $\mathbf{v}\in\mathscr S_m$ with $ c(\mathbf{v})\subseteq S$ evaluates to $abba$.
\end{lemma}
\begin{proof}
Choose $x\in Z_m\setminus S$; then $x\in P_m$. Every indicated $\mathbf{v}$ avoids $x$, so $\theta_x$ from Lemma~\ref{ind:lem:deletion} evaluates it to $abba$.
\end{proof}

\begin{lemma}\label{ind:lem:sparse}
Let $\mathbf{p}\approx \mathbf{p}+\mathbf{w}$ be a valid absorption identity of $\widehat A_2$ using at most $k$ source variables. Let $\beta$ be an arbitrary substitution into $\PX$. Fix actual short words $\mathbf{s},\mathbf{t}$, either of which may be formally omitted, and put $d=\ell(\mathbf{s})+\ell(\mathbf{t})$. Assume
\begin{equation}\label{ind:eq:sparse}
\varnothing\ne \mathbf{s}\mathbf{p}^\beta\mathbf{t}\subseteq\mathscr S_m,
\qquad \mathbf{s}\mathbf{w}^\beta\mathbf{t}\ne\varnothing,
\qquad |Z_m|>4k.
\end{equation}
Then $\ell(\mathbf{w})\le4$. Furthermore, every variable occurring in $\mathbf{w}$ is substituted by a nonempty set consisting entirely of words of length at most $4-d$; no long support class occurs in those substituted values.
\end{lemma}
\begin{proof}
\emph{1. Remove empty substitution images.}
Since $\varnothing\ne \mathbf{s}\mathbf{p}^\beta\mathbf{t}\subseteq\mathscr S_m$,
all its words have length three or four and $d\le3$.
For a formally omitted multiplier, its length is counted as
zero and its content as empty; this is notation only, not
a multiplicative identity adjoined to the semiring.
Put
\[
Y_0=\{y:\beta(y)=\varnothing\},
\]
and delete from $\mathbf{p}$ all monomials containing a variable in
$Y_0$; denote the nonempty remaining sum by $\mathbf{p}_0$.
The condition $\mathbf{s}\mathbf{w}^\beta\mathbf{t}\ne\varnothing$ in
\eqref{ind:eq:sparse} implies
$c(\mathbf{w})\cap Y_0=\varnothing$.  Assigning the new zero to all
variables of $Y_0$ in the valid source identity gives
\[
\widehat A_2\models \mathbf{p}_0\approx \mathbf{p}_0+\mathbf{w}.
\]
Every $y\in c(\mathbf{p}_0)$ has $\beta(y)\ne\varnothing$.

\emph{2. Bound every substituted word.}
Fix $y\in c(\mathbf{p}_0)$ and any $\mathbf{v}\in\beta(y)$.  Choose a
surviving monomial of $\mathbf{p}_0$ containing $y$ and a member of
the nonempty image for each of its other variable
occurrences.  The resulting product belongs to
$\mathbf{s}\mathbf{p}^\beta\mathbf{t}\subseteq\mathscr S_m$.  A long-support class
among its factors would produce a long class, while a
genuine word of length greater than $4-d$ would give a
product longer than four (or a long class).  Both contradict
the displayed containment.  Thus, for
\emph{every} $y\in c(\mathbf{p}_0)$,
\[
\varnothing\ne\beta(y)
\subseteq\{\mathbf{v}\in Z_m^+:1\le\ell(\mathbf{v})\le4-d\}.
\]
This applies to all members of each substitution image, not
only to the factors of one chosen expansion.

\emph{3. Construct a common old evaluation.}
For each $y\in c(\mathbf{p}_0)$ choose one fixed
$\mathbf{v}_y\in\beta(y)$ and use it at \emph{every} occurrence of
$y$.  Let
\[
S=c(\mathbf{s})\cup c(\mathbf{t})\cup\bigcup_{y\in c(\mathbf{p}_0)}c(\mathbf{v}_y).
\]
As $|c(\mathbf{p}_0)|\le k$ and each chosen word $\mathbf{v}_y$
contains at most $4-d$ variable occurrences, while
$|c(\mathbf{s})\cup c(\mathbf{t})|\le d$, we have
\[
|S|\le d+|c(\mathbf{p}_0)|(4-d)\le d+k(4-d)
\le4k<|Z_m|.
\]
The last inequality uses $0\le d\le3$ and $k\ge1$;
no distinctness among the letters of different $\mathbf{v}_y$
is assumed.
Every surviving monomial $\mathbf{u}$ of $\mathbf{p}_0$ yields a nonempty
expansion $\mathbf{s}\,\mathbf{u}(\mathbf{v}_y)\,\mathbf{t}\in\mathscr S_m$, so $a\in S$.
Lemma~\ref{ind:lem:realization} now provides an assignment
$\alpha:Z_m\to\widehat A_2$ under which all these expansions
have value $abba$.  The actual factors among $\mathbf{s},\mathbf{t},\mathbf{v}_y$ therefore take
old non-top values.  Set $\gamma(y)=\mathbf{v}_y^\alpha$ for
each $y\in c(\mathbf{p}_0)$.  For every monomial $\mathbf{u}$ of $\mathbf{p}_0$,
\[
\mathbf{s}^\alpha \mathbf{u}^\gamma \mathbf{t}^\alpha=abba.
\]
These are non-top equalities of concatenated words.
Cancelling the same left and right factors yields one old
non-top word $h$ such that
\[
\mathbf{u}^\gamma=h\quad(\mathbf{u}\text{ a monomial of }\mathbf{p}_0),
\qquad \mathbf{p}_0^\gamma=h\ne\infty.
\]
The consistency of the choice $\mathbf{v}_y$ at all occurrences is
essential for this common assignment $\gamma$.

\emph{4. Bound the query word.}
If some $y\in c(\mathbf{w})\setminus c(\mathbf{p}_0)$ existed, extend
$\gamma$ by assigning $y=\infty$ and arbitrary old
nonabsorbing factors to the remaining variables outside
$c(\mathbf{p}_0)$.  Then
\[
\mathbf{p}_0^\gamma=h\ne\infty,\qquad \mathbf{w}^\gamma=\infty,
\]
contrary to $\mathbf{p}_0\approx \mathbf{p}_0+\mathbf{w}$.
Consequently $c(\mathbf{w})\subseteq c(\mathbf{p}_0)$.  All variables of
$\mathbf{w}$ therefore take old non-top factor values under
$\gamma$.  If $\ell(\mathbf{w})\ge5$, their product is $\infty$,
again contradicting $\mathbf{p}_0^\gamma=h\ne\infty$.  Hence
\[
\ell(\mathbf{w})\le4,\qquad
\varnothing\ne\beta(y)
\subseteq\{\mathbf{v}\in Z_m^+:1\le\ell(\mathbf{v})\le4-d\}
\quad(y\in c(\mathbf{w})),
\]
as required.  Both bounds will be used in
Theorem~\ref{ind:thm:bound}.
\end{proof}

\begin{theorem}\label{ind:thm:bound}
For $m\ge2,k\ge1$, suppose
\[
|Z_m|=2m+2>\max\{4k,16\}.
\]
No identity valid in $\widehat A_2$ using at most $k$ variables can, by arbitrary polynomial substitution and arbitrary polynomial context, induce an elementary congruence step from $F_m$ to $(F_m+\lambda_{Z_m})$.
\end{theorem}
\begin{proof}
Suppose that such a crossing step exists. By
Proposition~\ref{ind:lem:normal} and
Lemma~\ref{ind:lem:contexts},
refine it into a finite chain of valid monomial-absorption
steps
\[
\mathbf{E}+\mathbf{M}\mathbf{p}^\beta\mathbf{N}\approx \mathbf{E}+\mathbf{M}(\mathbf{p}+\mathbf{w})^\beta \mathbf{N},
\]
where $\mathbf{p}\approx \mathbf{p}+\mathbf{w}$ has at most $k$ source variables.
Every intermediate expression remains in
$[\mathbf{Q}_m]_{\sim_m}$, so at least one refined step crosses
between $F_m$ and $F_m+\lambda_{Z_m}$.  Orient that step
from the part without $\lambda_{Z_m}$ to the other part.
By Theorem~\ref{ind:thm:fiber},
\[
\mathbf{M}\mathbf{p}^\beta\mathbf{N}\subseteq\mathscr S_m,
\qquad \lambda_{Z_m}\in \mathbf{M}\mathbf{w}^\beta\mathbf{N}.
\]

\emph{The old-side product is nonempty.}
A formally omitted multiplier contributes no factor;
$\varnothing\in\PX$, by contrast, annihilates products.
The second displayed condition
therefore ensures that every actual multiplier is nonempty.
If $\mathbf{M}\mathbf{p}^\beta\mathbf{N}=\varnothing$, then $\mathbf{p}^\beta=\varnothing$
and $\mathbf{w}^\beta\ne\varnothing$.  But the map
\[
\PX\longrightarrow\{0,\infty\}\le\widehat A_2,
\qquad
U\longmapsto
\begin{cases}0,&U=\varnothing,\\
\infty,&U\ne\varnothing,
\end{cases}
\]
is a semiring homomorphism: the union of sets and the
product of nonempty subsets of $\overline{X^+}$ preserve
nonemptiness.  Applying this map after $\beta$ to the
valid identity $\mathbf{p}\approx \mathbf{p}+\mathbf{w}$ would give
$0=\infty$, a contradiction.  Thus
\[
\mathbf{M}\mathbf{p}^\beta\mathbf{N}\ne\varnothing.
\]

\emph{Fix a witness and apply the sparse bound.}
Choose a \emph{single} product witness for
$\lambda_{Z_m}\in \mathbf{M}\mathbf{w}^\beta\mathbf{N}$, taking its left and right
factors to be $\mathbf{s},\mathbf{t}$ (with a formally omitted side absent).
Fixing these same factors, $\mathbf{p}^\beta\ne\varnothing$
ensures
\[
\varnothing\ne \mathbf{s}\mathbf{p}^\beta\mathbf{t}\subseteq\mathscr S_m,
\qquad \lambda_{Z_m}\in \mathbf{s}\mathbf{w}^\beta\mathbf{t}.
\]
No chosen multiplier can be a long-support class, since the
first product consists of short words.  Write
$d=\ell(\mathbf{s})+\ell(\mathbf{t})$.  Lemma~\ref{ind:lem:sparse} gives
$\ell(\mathbf{w})\le4$ and bounds each substituted word occurring in
an expansion of $\mathbf{w}$ by $4-d$.  The concrete concatenation
representing the chosen full-support class will be denoted by $\mathbf{v}$.
Its length satisfies
\[
\ell(\mathbf{v})
\le d+\ell(\mathbf{w})(4-d)
\le d+4(4-d)=16-3d\le16.
\]
On the other hand, a word in $\lambda_{Z_m}$ contains all
$|Z_m|>16$ distinct variables; its length is at least
$|Z_m|$.  This contradiction rules out the crossing step.
\end{proof}

\subsection{Finite countermodels and consequences}
For each $k\ge1$, take
\[
m(k)=\max\{2k,8\},\qquad X_k=Z_{m(k)},\qquad D_k=D_{k,X_k}.
\]
We have $|X_k|=2m(k)+2>\max\{4k,16\}$.

\begin{theorem}\label{ind:thm:finite}
For every $k\ge1$, the finite idempotent semiring $D_k$ satisfies
\[
D_k\models\Id_{\le k}(\widehat A_2),\qquad D_k\notin\W_\infty.
\]
\end{theorem}
\begin{proof}
By Proposition~\ref{ind:lem:local},
$D_k\models\Id_{\le k}(\widehat A_2)$.
If $D_k\in\W_\infty$, then Lemma~\ref{ind:lem:global}
would yield
\[
D_k\models Q_{m(k)}\approx Q_{m(k)}+K_{m(k)}.
\]
Evaluating at the canonical singleton generators gives
\[
[Q_{m(k)}]_{\Theta_{k,X_k}}
=[Q_{m(k)}+K_{m(k)}]_{\Theta_{k,X_k}}.
\]
By Lemma~\ref{ind:lem:crossing}, this requires an elementary
crossing from $F_{m(k)}$ to
$F_{m(k)}+\lambda_{Z_{m(k)}}$, contrary to
Theorem~\ref{ind:thm:bound}.  Hence $D_k\notin\W_\infty$.
\end{proof}

We now complete the proof of Theorem~\ref{ind:thm:main}.
\begin{proof}
Suppose, more generally, that $\mathcal K$ has an identity
basis $\Sigma$ (possibly infinite) whose members each involve at
most $k$ variables, and that
$\W_2\subseteq\mathcal K\subseteq\W_\infty$.
As $\W_2\subseteq\mathcal K$, every identity of $\Sigma$
holds in $\widehat A_2$. By Theorem~\ref{ind:thm:finite},
$D_k\models\Sigma$, whence $D_k\in\mathcal K\subseteq\W_\infty$,
contrary to the construction of $D_k$.

For $n\ge2$, the central four-letter block of $\mword_n$ is $t_{n-1}t_nt_nt_{n-1}$. Its nonempty factors, together with the old top and the new zero, form a subsemiring of $\widehat A_n$ isomorphic to $\widehat A_2$: both letters of the central block occur nowhere outside this block in $\mword_n$, so no new concatenation factor involving these letters arises. Consequently $\W_2\subseteq\W_n\subseteq\W_\infty$.
\end{proof}

The individual finite-basis classification in
Theorem~\ref{thm:main} follows directly from the theorem of Shao--Gao.
\begin{proof}
Shao--Gao~\cite[Theorem~1.1]{ShaoGao26} proved that the
zero extension $S(W)^0$ of a flat word semiring is finitely
based in the binary signature if and only if every defining word
has length at most two. Since $|\mword_n|=2n$, the classification
follows immediately.
\end{proof}

\begin{corollary}[Shao--Gao]\label{cor:individual-rank}
For every $n\ge2$, the variety $\W_n$ has infinite axiomatic rank.
\end{corollary}
\begin{proof}
The word set $\{\mword_n\}$ is finite and contains a word of
length $2n\ge4$, so the cited corollary applies.
\end{proof}

\begin{corollary}\label{cor:adjacent-all-NFB}
For every $n\ge2$, every variety in $[\W_n,\W_{n+1}]$ has
infinite axiomatic rank, and hence is nonfinitely based. The interval $[\W_1,\W_2]$ contains both a
finitely based variety and a nonfinitely based variety.
\end{corollary}
\begin{proof}
The first assertion is a consequence of
Theorem~\ref{ind:thm:main}. For the second, Theorem~\ref{thm:main}
shows that $\W_1$ is finitely based whereas $\W_2$ is not.
\end{proof}

\begin{corollary}[Lower bound on identity-variable complexity]\label{ind:cor:complexity}
Let
\[
\rho(m)=\min\{k\ge1:[\mathbf{Q}_m]_{\Theta_{k,Z_m}}=[\mathbf{Q}_m+\mathbf{K}_m]_{\Theta_{k,Z_m}}\}.
\]
For $m\ge8$,
\[
\left\lceil\frac{m+1}{2}\right\rceil\le\rho(m)\le 2m+2,
\qquad\text{and thus }\rho(m)\longrightarrow\infty.
\]
\end{corollary}
\begin{proof}
The identity $\delta_m$ itself uses $2m+2$ variables and is valid in $\widehat A_2$, giving the upper bound. For $k<\lceil(m+1)/2\rceil$ one has $4k<2m+2$, and for $m\ge8$ also $16<2m+2$. Theorem~\ref{ind:thm:bound}, combined with Lemma~\ref{ind:lem:crossing}, excludes equality in the defining set for such $k$.
\end{proof}

\begin{remark}[Scope of the construction]
For every fixed $k$, the constructed quotient is finite, while the
separating identity $\delta_m$ has unbounded variable support.
The argument does not assume that every subvariety of $\W_\infty$
is nonfinitely based: the proved interval has lower endpoint $\W_2$.
\end{remark}

\section{Variety inclusions and continuum intervals}\label{sec:variety}

We first establish the strict ascending chain of zero-extension
varieties and then obtain continuum-sized consecutive intervals
by an independent graph construction. The intervals between a
flat word variety and its zero extension are treated separately,
using the cyclic semirings of Shao--Gao and checking their
endpoint conditions.

\subsection{The chain of mirror-word varieties}

\begin{proposition}\label{lem:chains}
For every $n\ge1$,
$\W_n\subsetneq\W_{n+1}$ and
$\V_n\subsetneq\W_n$.
\end{proposition}
\begin{proof}
The central mirror block embeds $A_n$ into $A_{n+1}$, and the
embedding extends to the new zero. The identity
\[
 z_1\cdots z_{n+1}z_{n+1}\cdots z_1
 \approx z_1\cdots z_{n+1}z_{n+1}\cdots z_1
      +z_1^3\cdots z_{n+1}^3
\]
holds in $\widehat A_n$: with a zero
variable both words are zero, and otherwise both have absorbing
old value. It fails in $\widehat A_{n+1}$ under $z_i=a_i$.
The old algebra is a subalgebra of its extension. The identity
$x^3\approx y^3$ holds in $A_n$ and fails in the extension at
$x=0,y=a_n$, so the containment is proper.
\end{proof}

\subsection{Graph constructions and separating identities}

We use finite context quotients to encode graph homomorphisms
in the consecutive intervals; compare~\cite{GaoRenZhao26}.
This is the restriction of the conjunctive syntactic semiring
congruence of a language, as used by
Pol\'ak~\cite{Polak03Syntactic}, to the free ai-semiring of
nonempty finite subsets of $X^+$. Thus we use its constant-free
positive-word fragment, rather than introducing a new operation.

Let $X$ be a finite alphabet and let $P\subseteq X^+$ be a finite,
nonempty language. Let $\mathcal P_f(X^+)$ be the free ai-semiring on $X$: its elements are
finite nonempty subsets of $X^+$, with union and setwise concatenation
as operations. For each nonempty finite $F\subseteq X^+$, define its context profile by
\[
 \rho_P(F)=\{(\mathbf{u},\mathbf{v})\in X^*\times X^*: \mathbf{u}F\mathbf{v}\subseteq P\}.
\]
Declare $F\equiv_P G$ when $\rho_P(F)=\rho_P(G)$.
For nonempty finite $F,G,C\subseteq X^+$ and contexts
$\mathbf{u},\mathbf{v}\in X^*$, addition is reflected in the profile identity
\[
 \rho_P(F\cup G)=\rho_P(F)\cap\rho_P(G).
\]
Moreover, $(\mathbf{u},\mathbf{v})\in\rho_P(FC)$ if and only if
$(\mathbf{u},\mathbf{c}\mathbf{v})\in\rho_P(F)$ for every $\mathbf{c}\in C$.
Similarly, $(\mathbf{u},\mathbf{v})\in\rho_P(CF)$ if and only if
$(\mathbf{u}\mathbf{c},\mathbf{v})\in\rho_P(F)$ for every $\mathbf{c}\in C$.
Thus $\equiv_P$ is a semiring congruence. Denote the quotient by
$D(P)$, with absorbing element $\infty$ represented by an empty
context profile. Put $L=\max\{\ell(\mathbf{p}):\mathbf{p}\in P\}$. If
$\rho_P(F)\ne\varnothing$, every $f\in F$ has length at most $L$
and every $(\mathbf{u},\mathbf{v})\in\rho_P(F)$ satisfies $\ell(\mathbf{u})+\ell(\mathbf{v})<L$.
There are therefore only finitely many profiles, and $D(P)$ is finite.
The empty word is used solely in contexts, not as a semiring constant.

In the notation of Section~\ref{sec:prelim}, for each finite
nonempty set $P=\{\mathbf{p}_1,\ldots,\mathbf{p}_r\}$ the corresponding term is
\[
 \mathbf{w}=\mathbf{p}_1+\cdots+\mathbf{p}_r.
\]
We use only this ordinary finite sum of words; its order is
immaterial, and the empty sum is not a term.

\begin{proposition}\label{lem:syntactic-membership}
Let $\mathcal K$ be a variety of ai-semirings and let
$P=\{\mathbf{p}_1,\ldots,\mathbf{p}_r\}\subseteq X^+$ be finite and nonempty.
Put $\mathbf{w}=\mathbf{p}_1+\cdots+\mathbf{p}_r$. Suppose that every $\mathbf{q}\in X^+$ satisfying
$\mathcal K\models \mathbf{q}\preceq \mathbf{w}$ belongs to $P$.
Then $D(P)\in\mathcal K$.
\end{proposition}
\begin{proof}
Let $\mathbf{s}\approx \mathbf{t}$ hold in $\mathcal K$, and substitute arbitrary
nonempty finite languages for its variables. Write $F,G$ for
the resulting nonempty word sets. Thus
\[
 \mathcal K\models F\approx G,
\]
where, as in Section~\ref{sec:prelim}, each finite word set is
identified with its additive term.

Fix $\boldsymbol{\alpha},\boldsymbol{\beta}\in X^*$ with $\boldsymbol{\alpha}F\boldsymbol{\beta}\subseteq P$. Context closure of
$\mathbf{s}\approx \mathbf{t}$, followed by additive idempotence, gives
\[
 \mathcal K\models
 \boldsymbol{\alpha}F\boldsymbol{\beta}\approx \boldsymbol{\alpha}G\boldsymbol{\beta},
 \qquad
 \mathbf{w}\approx \mathbf{w}+\boldsymbol{\alpha}F\boldsymbol{\beta},
 \qquad
 \mathbf{w}\approx \mathbf{w}+\boldsymbol{\alpha}G\boldsymbol{\beta}.
\]
For every $\mathbf{g}\in G$ we obtain
$\mathcal K\models\boldsymbol{\alpha}\mathbf{g}\boldsymbol{\beta}\preceq \mathbf{w}$.
The hypothesis therefore gives $\boldsymbol{\alpha}\mathbf{g}\boldsymbol{\beta}\in P$.
Hence $\boldsymbol{\alpha}G\boldsymbol{\beta}\subseteq P$. Interchanging $F,G$ proves the
reverse implication, so $\rho_P(F)=\rho_P(G)$.
Every identity valid in $\mathcal K$ is therefore valid in
$D(P)$, whence $D(P)\in\mathcal K$.
\end{proof}

For a finite simple graph $G$ without isolated vertices and with at
least one edge, the following constructions are defined using a
variable $y_v$ for each vertex $v\in V(G)$. The anchor and vertex symbols in these languages are formal variables;
letters in the fixed defining words are distinguished by context.

For $n\ge2$, put $h=2n$. Choose $h-2$ distinct anchor
letters indexed by
\[
 a_1,\ldots,a_{n-1},\quad a_{n+2},\ldots,a_h,
\]
and a distinct letter $y_v$ for each vertex of $G$. For every oriented edge $(u,v)$, form the word
of length $h$ with $y_u,y_v$ in positions $n,n+1$ and the anchors
in their indicated positions. Include also all words obtained by
independently interchanging the letters in positions
$(j,h+1-j)$ for $2\le j\le n-1$. The pair of outer positions
$(1,h)$ is never interchanged. Let $P_{n,G}$ be this finite set
of linear words and put
\[
 T_{n,G}=D(P_{n,G}),\qquad
 \mathbf{u}_{n,G}=\mathbf{p}_1+\cdots+\mathbf{p}_r
 \quad\text{for }P_{n,G}=\{\mathbf{p}_1,\ldots,\mathbf{p}_r\}.
\]

For $n=1$, put
\[
 \begin{aligned}
 P_{1,G}&=\{a y_u y_v:(u,v)\text{ is an oriented edge of }G\},\\
 T_{1,G}&=D(P_{1,G}),\qquad
 \mathbf{u}_{1,G}=\mathbf{p}_1+\cdots+\mathbf{p}_r
 \quad\text{for }P_{1,G}=\{\mathbf{p}_1,\ldots,\mathbf{p}_r\}.
 \end{aligned}
\]

\begin{proposition}\label{lem:adjacent-graph}
For every $n\ge1$, the finite semiring $T_{n,G}$ belongs to
$\W_{n+1}$. For every
finite simple graph $H$ without isolated vertices and with at least
one edge, there is an identity $\beta_{n,H}$ that
holds in $\W_n$.
It holds in $T_{n,G}$ if and only if there is no graph
homomorphism from $H$ to $G$.
\end{proposition}
\begin{proof}
\emph{1. Membership for $n\ge2$.}
Put
\[
 \mathbf{w}_n=b_1\cdots b_nb_n\cdots b_2c,
\]
where $c$ is distinct from the $b_i$, and define the
nonerasing morphism
\[
 \varphi(b_1)=a_1a_2,\qquad
 \varphi(b_j)=a_{j+1}\ (2\le j\le n),\qquad
 \varphi(c)=a_2a_1.
\]
The codewords have distinct initial letters, so $\varphi$ is
injective, and $\varphi(\mathbf{w}_n)=\mword_{n+1}$. Moreover,
every occurrence of a codeword in $\mword_{n+1}$ starts and
ends at a block boundary: $a_1a_2$ occurs only at the
beginning, $a_2a_1$ only at the end, and the remaining
codewords are single letters. Consequently, the word $\mathbf{t}$ is a factor of $\mathbf{w}_n$
if and only if $\varphi(\mathbf{t})$ is a factor of $\mword_{n+1}$.
In particular $\varphi$ preserves factor and nonfactor
products, while injectivity preserves the flat addition.
The map on factors, extended by $\infty\mapsto\infty$
and then by $0\mapsto0$, gives embeddings
\[
 S(\mathbf{w}_n)\hookrightarrow A_{n+1},\qquad
 S(\mathbf{w}_n)^0\hookrightarrow\widehat A_{n+1}.
\]

We apply Proposition~\ref{lem:syntactic-membership} with
$P=P_{n,G}$ and
$\mathcal K=\W_n\vee\Var(S(\mathbf{w}_n))$.
Suppose that
\[
 \mathcal K\models \mathbf{q}\preceq \mathbf{u}_{n,G}
 \qquad(\mathbf{q}\in X^+).
\]
We prove $\mathbf{q}\in P$ by five restrictions.

\smallskip\noindent\emph{(i) Variables.}
Fold the paired anchors to the corresponding letters of
$\mword_n$ and all vertex variables to its central letter.
Each summand of $\mathbf{u}_{n,G}$ then evaluates to the proper
factor $\mword_n$. Assigning $\infty$ to any extra
variable of $\mathbf{q}$ would violate absorption; hence
\[
 c(\mathbf{q})\subseteq c(P).
\]

\smallskip\noindent\emph{(ii) Edge support.}
Since $D_2\in\W_n$, assign $1$ precisely to the variables
in $c(\mathbf{q})$, and $0$ to the others. The value of $\mathbf{q}$ is $1$;
absorption forces some $\mathbf{r}\in P$ to take value $1$. Thus
\[
 c(\mathbf{r})\subseteq c(\mathbf{q}),\qquad |c(\mathbf{r})|=h=2n.
\]
Here $\mathbf{r}$ contains the $h-2$ distinct anchors and the two
distinct endpoints of an edge.

\smallskip\noindent\emph{(iii) Length.}
Under the folding assignment in (i), absorption forces $\mathbf{q}$
to have the same proper value $\mword_n$. Each variable
has a one-letter image, so $\ell(\mathbf{q})=h$. Together with (ii),
this says that $\mathbf{q}$ is a linear word consisting of exactly
the $h$ letters of some $\mathbf{r}\in P$.

\smallskip\noindent\emph{(iv) Internal positions.}
Equality of the folded word with $\mword_n$ restricts every
letter of $\mathbf{q}$ to one of the two mirror positions associated
with its letter of $\mathbf{r}$. The defining language $P$ includes
all exchanges of internal pairs $2\le j<n$ and both
orientations of the central edge.

\smallskip\noindent\emph{(v) Outer positions.}
In $S(\mathbf{w}_n)$, assign the outer anchors $b_1,c$, each internal
mirror pair $b_j$, and all vertex variables $b_n$.
Every member of $P$ evaluates to $\mathbf{w}_n$. If the two outer
anchors of $\mathbf{q}$ were interchanged, its value would begin
with $c$ and end with $b_1$, hence be $\infty$, contrary
to the assumed absorption. Thus the outer pair is fixed,
and (i)--(v) give $\mathbf{q}\in P$.
Proposition~\ref{lem:syntactic-membership} now yields
\[
 T_{n,G}\in\W_n\vee\Var(S(\mathbf{w}_n))
 \subseteq\W_{n+1}.
\]

\medskip\noindent\emph{2. Separation for $n\ge2$.}
Choose $\mathbf{t}_H\in P_{n,H}$ and interchange its two outer
letters to obtain $\mathbf{q}_H$. Define the absorption identity
\[
 \beta_{n,H}:\qquad
 \mathbf{u}_{n,H}\approx \mathbf{u}_{n,H}+\mathbf{q}_H.
\]
We first show $\W_n\models\beta_{n,H}$. Fix an
evaluation in $\widehat A_n$ and let $f$ be the value of
$\mathbf{u}_{n,H}$.
If $f=\infty$, absorption is immediate. If $f=0$,
then the summand $\mathbf{t}_H$ has value $0$, so one of its
variables takes the new zero. Since
$c(\mathbf{q}_H)=c(\mathbf{t}_H)$, the value of $\mathbf{q}_H$ is also $0$.
Finally, suppose $f$ is a proper old factor. Then $\mathbf{t}_H$
has the same proper value. Its $h=2n$ factors are all
nonempty old factors, and their concatenation is a factor
of $\mword_n$ of length at most $2n$. Hence each factor
is a single letter and their product is exactly
$\mword_n$. The outer letters coincide, so their
interchange leaves the product unchanged. In all cases,
$\beta_{n,H}$ holds in $\widehat A_n$.

We next prove that $T_{n,G}$ fails $\beta_{n,H}$ if and
only if a graph homomorphism from $H$ to $G$ exists.
If $\psi:H\to G$ is a graph homomorphism, substitute
singleton anchor letters and
$y_v\mapsto y_{\psi(v)}$. Every summand of
$\mathbf{u}_{n,H}$ belongs to $P_{n,G}$, whereas $\mathbf{q}_H$ with its
reversed outer pair does not. Since all these words have
length $h$, the empty-context pair belongs to the premise
profile but not to the query profile. Thus
$T_{n,G}\not\models\beta_{n,H}$.

Conversely, suppose $\beta_{n,H}$ fails in $D(P_{n,G})$.
Lift the witnessing valuation to nonempty finite word sets.
Because addition in the quotient is intersection of context
profiles, there exist $\mathbf{u},\mathbf{v}\in X^*$ such that
\[
 \mathbf{u}F\mathbf{v}\subseteq P_{n,G},\qquad
 \mathbf{u}Q\mathbf{v}\not\subseteq P_{n,G},
\]
where $F$ is the expanded premise language and $Q$ is the
expanded query language. Every word of $P_{n,G}$ has length
$h$, while each expanded premise monomial contains $h$
nonempty word factors. Therefore $\mathbf{u}=\mathbf{v}=\varepsilon$, and
every word in each substituted language occurring in a
premise monomial has length one. Since $H$ has no isolated
vertices, this applies to every vertex variable.
The positional restrictions in $P_{n,G}$ must also be used.
For each position $1\le i\le h$, let $\Lambda_i$ be the set of
letters that occur in position $i$ among the words of
$P_{n,G}$. For a variable $z$ of the premise, let $J_z$
be the set of positions occupied by $z$ in some defining
word of $P_{n,H}$. Since \emph{every} word obtained by
independently choosing one letter from each substituted
language lies in $P_{n,G}$, the (nonempty) one-letter
language $L_z$ assigned to $z$ satisfies
\[
       L_z\subseteq\bigcap_{i\in J_z}\Lambda_i.
\]
The outer anchors therefore have their prescribed outer
letters as their only possible images; internal anchors
have images among the anchor letters permitted in their
mirror-position pair. The two central positions admit only
vertex letters of $G$. Every vertex variable of $H$ occurs
in an edge word, since $H$ has no isolated vertices; hence
its image language contains only vertex letters of $G$.
The latter conclusion uses full-language containment as
well as the positional restriction, not merely the length
bound.

The language substituted for $y_w$ thus consists of vertex
letters of $G$. Let $B_w\ne\varnothing$ be their set of
vertex indices. For each edge $\{w,w'\}\in E(H)$ and \emph{every} choice
of $b\in B_w$ and $b'\in B_{w'}$, choose these letters in
the corresponding expanded premise monomial and choose any
letters from the nonempty anchor languages. Full-language
containment ensures that each resulting word belongs to
$P_{n,G}$, not merely that one such choice belongs to it.
Its two central letters therefore form an oriented edge of
$G$. Both orientations of each source edge are included,
so
\[
 B_w\times B_{w'}\subseteq E^{\mathrm{or}}(G),
 \qquad
 E^{\mathrm{or}}(G)
 =\{(z,z'):\{z,z'\}\in E(G)\}.
\]
Choose $\psi(w)\in B_w$ for every $w\in V(H)$.
The displayed containment implies
$\{\psi(w),\psi(w')\}\in E(G)$ whenever
$\{w,w'\}\in E(H)$. Thus $\psi:H\to G$.
The argument allows arbitrary nonempty finite-language
substitutions, not merely singleton valuations.

\medskip\noindent\emph{3. Membership for $n=1$.}
We apply Proposition~\ref{lem:syntactic-membership} directly
in $\W_2$. Suppose
\[
 \widehat A_2\models \mathbf{q}\preceq \mathbf{u}_{1,G}.
\]
Assign the anchor $a$ to the old letter $a$ and every
vertex variable to $b$. The premise then has proper value
$abb$. Assigning $\infty$ to an extra variable of $\mathbf{q}$
shows
\[
 c(\mathbf{q})\subseteq c(P_{1,G}).
\]
Next apply the homomorphism $\widehat A_2\to D_2$
sending the new zero to $0$ and each old element to $1$.
Assign $1$ precisely to $c(\mathbf{q})$, and $0$ elsewhere.
Absorption forces some $\mathbf{r}=ay_uy_v\in P_{1,G}$ to satisfy
\[
 c(\mathbf{r})\subseteq c(\mathbf{q}).
\]
Under the all-old assignment $a\mapsto a$, $y_v\mapsto b$,
the value of $\mathbf{q}$ must be $abb$. Therefore $\mathbf{q}$ has length
three, contains exactly the distinct variables of $\mathbf{r}$,
and begins with $a$. Since both orientations of every
edge are in $P_{1,G}$, it follows that $\mathbf{q}\in P_{1,G}$;
hence $T_{1,G}\in\W_2$.

\medskip\noindent\emph{4. Separation for $n=1$.}
Choose an oriented edge $(u,v)$ of $H$, let
$\mathbf{q}_H=y_uay_v$, and define
\[
 \beta_{1,H}:\qquad
 \mathbf{u}_{1,H}\approx \mathbf{u}_{1,H}+\mathbf{q}_H.
\]
This identity holds in $\W_1$. When the premise is
$\infty$, absorption is immediate. When it is $0$, the
summand $ay_uy_v$ and $\mathbf{q}_H$ are both $0$. The premise
cannot take a proper old value, since every product of
three old elements in $\widehat A_1$ equals $\infty$.

A homomorphism $\psi:H\to G$ gives a counterevaluation
in $D(P_{1,G})$ by substituting singleton letters, since
the premise has a nonempty context profile while
$\mathbf{q}_H$, beginning with a vertex letter, has empty profile.
Conversely, lift any counterevaluation to nonempty finite
languages. The common length three of the defining words
forces an empty context and one-letter factors in all
expanded premise words. The first position of every word in $P_{1,G}$ is $a$,
whereas the other two positions contain only vertex letters;
thus the anchor language consists only of $a$. Since every
vertex of $H$ occurs in an oriented premise word, its
one-letter language consists entirely of vertex letters
of $G$. For each source edge and every choice of one
letter from either endpoint language, the corresponding
expanded premise word lies in $P_{1,G}$. Consequently the
Cartesian product of the two endpoint index sets is
contained in $E^{\mathrm{or}}(G)$, not merely one selected
pair. Choosing one index from each nonempty vertex language
defines the required homomorphism $H\to G$.
This proves both statements when $n=1$.
\end{proof}

\subsection{Continuum intervals}

\begin{theorem}[Consecutive mirror-word intervals]\label{thm:adjacent-cardinality}
For every $n\ge1$, the interval $[\W_n,\W_{n+1}]$ has
cardinality $2^{\aleph_0}$ and contains an order-embedded
copy of $(\mathcal P(\mathbb N),\subseteq)$.
\end{theorem}
\begin{proof}
Choose finite simple graphs $G_0,G_1,\ldots$ without isolated
vertices such that both their chromatic numbers and their odd
girths strictly increase, with all chromatic numbers at least
three and all odd girths at least five. Their existence follows
from Erd\H{o}s's theorem~\cite{Erdos59}. The graphs form an
antichain for graph homomorphisms: a homomorphism cannot
increase chromatic number, and it cannot map an odd cycle
into a graph of larger odd girth.

For $J\subseteq\mathbb N$ set
\[
 \mathcal K_{n,J}
   =\W_n\vee\bigvee_{j\in J}\Var(T_{n,G_j}).
\]
By Proposition~\ref{lem:adjacent-graph},
\[
 \W_n\subseteq\mathcal K_{n,J}\subseteq\W_{n+1}.
\]
The identity $\beta_{n,G_i}$ holds in $\mathcal K_{n,J}$
if and only if $i\notin J$. Consequently, for subsets
$J,J'\subseteq\mathbb N$, the inclusion
$J\subseteq J'$ holds if and only if
$\mathcal K_{n,J}\subseteq\mathcal K_{n,J'}$.
Thus $J\mapsto\mathcal K_{n,J}$ is an order embedding of
$(\mathcal P(\mathbb N),\subseteq)$ into the interval.
This gives the lower bound $2^{\aleph_0}$. The binary
semiring signature has countably many identities, hence
at most $2^{\aleph_0}$ equational theories, giving the
matching upper bound.
\end{proof}

\subsection{Zero-extension intervals via cyclic semirings}

We apply the cyclic semiring family $E_p$ and separating identities
of Shao--Gao~\cite[Sections~10--11]{ShaoGao26}, rather than
repeating a second graph construction. Their letter notation for
linear words is kept distinct from our mirror-word notation.
For every integer $p\ge2$, with indices taken modulo $p$, put
\[
 \mathbf{c}_p=\sum_{i\in\mathbb Z/p\mathbb Z}
  (x_i y_i z_0+x_i y_{i+1}z_1),\qquad
 \mathbf{q}_p=x_0y_0z_1,
\]
and write
\[
 \varepsilon_p:\qquad \mathbf{c}_p\approx \mathbf{c}_p+\mathbf{q}_p.
\]
By Shao--Gao~\cite[Proposition~10.2]{ShaoGao26}, for
$m,p\ge2$ the semiring $E_m$ satisfies
$\varepsilon_p$ if and only if $m\nmid p$.

\begin{lemma}[Endpoint verification]\label{lem:SG-cyclic-endpoints}
For every $n\ge2$ and $p\ge2$,
\[
 E_p\in\W_n,\qquad A_n\models\varepsilon_p.
\]
\end{lemma}
\begin{proof}
We begin with $E_p\in\W_n$. This requires checking
the mirror-word endpoint rather than directly transferring
the linear-word result.
The three-word permutation identities
\[
 xyz\approx yxz,\qquad xyz\approx xzy,\qquad
 xyz\approx zyx
\]
all fail in $\widehat A_2=S(abba)^0$. Indeed, with
$(x,y,z)=(a,b,b)$, the nonabsorbing product $abb$ differs
from $bab$ and $bba$, while with $(x,y,z)=(b,b,a)$ the
nonabsorbing product $bba$ differs from $bab$. Thus
Shao--Gao~\cite[Theorem~23.2]{ShaoGao26} yields
$\Var(S(abc))\subseteq\W_2$ for distinct letters $a,b,c$.
Their Theorem~22.2 then gives
\[
 \Var(S(abc)^0)\subseteq\W_2.
\]
Their cyclic construction places
$E_p\in\Var(R_3)\subseteq\Var(S(abc)^0)$
(see~\cite[Proposition~10.1 and Section~5]{ShaoGao26}).
Consequently $E_p\in\W_2\subseteq\W_n$, where the last
inclusion follows from Proposition~\ref{lem:chains}.

\medskip\noindent We next verify $A_n\models\varepsilon_p$.
Evaluate $\mathbf{c}_p$ in $A_n=S(\mword_n)$. When its value is
$\infty$, the absorption $\varepsilon_p$ is automatic.
Otherwise all its cubic summands have one and the same value,
a nonempty factor of $\mword_n$. In particular, the assigned
values $U_i,V_i,Z_0,Z_1$ of $x_i,y_i,z_0,z_1$ are nonempty
words, and
\[
 U_iV_iZ_0=U_iV_{i+1}Z_1
 \quad(i\in\mathbb Z/p\mathbb Z).
\]
Taking word lengths gives
$|V_{i+1}|-|V_i|=|Z_0|-|Z_1|$.
Summing around the cycle shows $|Z_0|=|Z_1|$ and hence
$|V_i|=|V_{i+1}|$ for all $i$. Cancellation in the free
semigroup, followed by comparison of these equal-length
prefixes and suffixes, gives $V_i=V_{i+1}$ and $Z_0=Z_1$.
It follows that $U_0V_0Z_1=U_0V_0Z_0$, which is an existing
summand value of $\mathbf{c}_p$. Thus $A_n\models\varepsilon_p$.
\end{proof}

\begin{corollary}[Zero-extension continuum intervals]\label{cor:zero-extension-continuum}
For every $n\ge2$, the interval $[\V_n,\W_n]$ has
cardinality $2^{\aleph_0}$ and contains an order-embedded
copy of $(\mathcal P(\mathbb N),\subseteq)$.
\end{corollary}
\begin{proof}
For each set $J$ of prime numbers, define
\[
 \mathcal J_{n,J}=\V_n\vee\bigvee_{p\in J}\Var(E_p).
\]
By Lemma~\ref{lem:SG-cyclic-endpoints} and
$A_n\le\widehat A_n$, each $\mathcal J_{n,J}$ belongs to
$[\V_n,\W_n]$. For every prime $p$, the same lemma and
Shao--Gao~\cite[Proposition~10.2]{ShaoGao26} show that
the identity $\varepsilon_p$ holds in $\mathcal J_{n,J}$
if and only if $p\notin J$.
Hence $J\mapsto\mathcal J_{n,J}$ is an order embedding.
The power set of the primes has cardinality $2^{\aleph_0}$,
and the general countable-signature bound gives the reverse
inequality.
\end{proof}

\begin{corollary}\label{cor:large-interval-continuum}
For every $n\ge2$, the interval $[\V_n,\W_\infty]$ has cardinality
$2^{\aleph_0}$ and contains a chain and an antichain of that
cardinality.
\end{corollary}
\begin{proof}
The interval $[\W_n,\W_{n+1}]$ is contained in
$[\V_n,\W_\infty]$. Apply Theorem~\ref{thm:adjacent-cardinality}
and the countable-signature upper bound.
\end{proof}

\begin{corollary}\label{cor:infinite-continuum}
The subvariety lattice of $\W_\infty$ has cardinality
$2^{\aleph_0}$.
\end{corollary}
\begin{proof}
The interval $[\W_2,\W_3]$ is contained in
the subvariety lattice of $\W_\infty$; the theorem gives
the lower bound and the countable-signature argument gives the
upper bound.
\end{proof}

\subsection{Local finiteness}\label{sec:infinite-join}

\begin{proposition}\label{prop:infinite-local}
The variety $\W_\infty$ is locally finite.
\end{proposition}
\begin{proof}
Fix an $r$-element variable set $X_r$. Only finitely many words over
$X_r$ contain each variable at most twice. For any other word $\mathbf{w}$ put
$D=c(\mathbf{w})$. In every $\widehat A_n$ we have
\[
 \mathbf{w}\approx\prod_{x\in D}x^3,
\]
with the order of factors on the right fixed arbitrarily. Indeed,
if some variable in $D$ takes the new value $0$, both sides are $0$;
otherwise both are $\infty$. Thus every term over $X_r$ is equivalent
to a nonempty sum of words from a finite set: the words in which each
variable occurs at most twice, together with one word
$\prod_{x\in D}x^3$ for each nonempty $D\subseteq X_r$.
Idempotence of addition now implies that the $r$-generated free
semiring of $\W_\infty$ is finite.
\end{proof}

\section{The unique limit subvariety of \texorpdfstring{$\W_1$}{W1}}\label{sec:unique-limit}
\subsection{Identities and normal forms in the four-element variety}\label{sec:positive}

Throughout this section, a finite nonempty set of words
$P=\{\mathbf{p}_1,\ldots,\mathbf{p}_r\}$ is represented by the ordinary ai-semiring
term $\mathbf{p}_1+\cdots+\mathbf{p}_r$, exactly as in
Section~\ref{sec:prelim}. We never use an empty additive term.

By Zhao--Wu~\cite{ZhaoWu25}, the variety
$\W_1=\Var(\widehat A_1)$, where
$\widehat A_1\cong S_c(a^2)^0$, is finitely based.
We now give a self-contained completeness proof for the following
six-identity basis in the binary ai-semiring signature.
The proof uses the absorption criterion in
Proposition~\ref{cor:absorption} and the normal forms below.

Set $b=a^2$. The operations on $\widehat A_1=\{0,a,b,\infty\}$ are
\[
\begin{array}{c|cccc}
 +&0&a&b&\infty\\\hline
 0&0&a&b&\infty\\
 a&a&a&\infty&\infty\\
 b&b&\infty&b&\infty\\
 \infty&\infty&\infty&\infty&\infty
\end{array}
\qquad
\begin{array}{c|cccc}
 \cdot&0&a&b&\infty\\\hline
 0&0&0&0&0\\
 a&0&b&\infty&\infty\\
 b&0&\infty&\infty&\infty\\
 \infty&0&\infty&\infty&\infty
\end{array}.
\]

\begin{proposition}[A six-identity equational basis]\label{lem:six-valid}
Relative to the ai-semiring axioms, the following six identities
constitute an equational basis of $\W_1$:
\begin{align}
 xy&\approx yx, \label{eq:basis-comm}\\
 xyz&\approx x^3yz, \label{eq:basis-cubic}\\
 x+x^2&\approx x^3, \label{eq:basis-square}\\
 x+xy&\approx x+xy+x^3y, \label{eq:basis-mixed}\\
 x^3+xy&\approx x^3, \label{eq:basis-absorb}\\
 x^2+y^2&\approx x^2+y^2+xy. \label{eq:basis-quadratic}
\end{align}
\end{proposition}

Each identity is verified directly from the four-element Cayley
tables of $\widehat A_1$. We derive the required word normal forms
before completing the proof of Proposition~\ref{lem:six-valid}.

First, we derive two additional identities from the six
identities above. By \eqref{eq:basis-square} and additive
idempotence,
\[
 x^2+x^3\approx x^2+(x+x^2)
 \approx x+x^2\approx x^3.
\]
Multiplying \eqref{eq:basis-square} by $x$ and distributing gives
$x^2+x^3\approx x^4$. Hence
\begin{equation}\label{eq:derived-power}
 x^3\approx x^4.
\end{equation}
Using \eqref{eq:basis-cubic} with $(x,y,z)=(x,x,t)$ and
then \eqref{eq:derived-power}, we obtain
\begin{equation}\label{eq:six-square-cube}
 x^2t\approx x^4t\approx x^3t.
\end{equation}
By commutativity, these relations give
\[
\begin{aligned}
 x^2y
 &\approx x^2y^3 &&\text{by \eqref{eq:basis-cubic} with }(y,x,x),\\
 &\approx x^3y^3 &&\text{by \eqref{eq:six-square-cube} with }t=y^3,\\
 &\approx x^3y^2 &&\text{by \eqref{eq:six-square-cube} with }x,y
                  \text{ interchanged, in reverse},\\
 &\approx xy^2 &&\text{by \eqref{eq:basis-cubic} with }(x,y,y),
                   \text{ in reverse}.
\end{aligned}
\]
Thus we may also use
\begin{equation}\label{eq:derived-square}
 x^2y\approx xy^2.
\end{equation}
The preceding reductions yield a convenient normal form.
Fix a total order on the variable set.

\begin{lemma}\label{lem:norm-four}
Modulo \eqref{eq:basis-comm}, \eqref{eq:basis-cubic},
\eqref{eq:derived-power}, and \eqref{eq:derived-square},
every word is equivalent either to a linear word $x$, to a
quadratic word $x^2$ or $xy$ with $x<y$, or to
\[
 x^3\prod_{z\in c(\mathbf{w})\setminus\{x\}}z,
\]
where $\ell(\mathbf{w})\ge3$ and $x$ is the least variable in $c(\mathbf{w})$.
An empty final product is omitted.
\end{lemma}
\begin{proof}
By commutativity, order the letters of $\mathbf{w}$. The identities
\eqref{eq:basis-cubic} (with two equal factors) and
\eqref{eq:derived-power} give
\[
 x^2y\approx x^3y,\qquad x^k\approx x^3\quad(k\ge3).
\]
Thus any word of length at least three has a cubic pivot:
use a repeated letter if one exists, and otherwise apply
\eqref{eq:basis-cubic} to a square-free word.
For each other letter, \eqref{eq:derived-square} gives
\[
 x^3y^2\approx x^4y\approx x^3y,
\]
so its multiplicity reduces to one. Finally,
\[
 x^3y\approx x^2y^2\approx xy^3
\]
(multiply \eqref{eq:derived-square} by $x$ and by $y$)
transfers the pivot to the least letter of $c(\mathbf{w})$.
The degree-one and degree-two forms require no reduction.
\end{proof}

\begin{proof}[Completion of the proof of Proposition~\ref{lem:six-valid}]
It remains to prove completeness. We work with normal forms from
Lemma~\ref{lem:norm-four}. For a nonempty finite variable set
$C$, let $c$ be its least element and set
\[
 \mathbf{t}_{C}=c^3\prod_{z\in C\setminus\{c\}}z.
\]
As usual, the final product is omitted when $C=\{c\}$.
The preliminary identities below hold modulo the six displayed
identities.

\smallskip\noindent
\emph{Absorption consequences.}
First, \eqref{eq:basis-square} and idempotence give
$c^3+c\approx c^3$. For every word $\mathbf{m}$ and every nonempty
$D\subseteq c(\mathbf{m})$, we claim
\[
 \mathbf{t}_{D}+\mathbf{m}\approx \mathbf{t}_{D}.
\]
Indeed, if $d$ is the least variable of $D$ and $\mathbf{v}$ lists
once each variable of $D\setminus\{d\}$, commutativity permits
$\mathbf{m}=d\mathbf{v}\mathbf{r}$ for some possibly empty word $\mathbf{r}$. Thus the claim
follows by multiplying $d^3+d\mathbf{r}\approx d^3$ by $\mathbf{v}$;
when $\mathbf{r}$ is omitted we use $d^3+d\approx d^3$, and otherwise
we use \eqref{eq:basis-absorb} with $y=\mathbf{r}$.
No multiplicative identity is introduced by an omitted word.

Second, if a term $\mathbf{u}$ has a linear summand $x$ and also a
quadratic summand $\mathbf{p}$ containing $x$, then
\[
 \mathbf{u}+\mathbf{t}_{c(\mathbf{p})}\approx \mathbf{u}.
\]
For $\mathbf{p}=x^2$, apply \eqref{eq:basis-square};
for $\mathbf{p}=xy$ with $x\ne y$, apply \eqref{eq:basis-mixed}
and normalize the added cubic word using
Lemma~\ref{lem:norm-four}. All remaining summands of $\mathbf{u}$
are unchanged.

\smallskip\noindent
\emph{Exhaustion of valid normal-word absorptions.}
Let $\mathbf{u}$ be a nonempty term whose summands are normal words,
and let $\mathbf{q}$ be a normal word. The case $\mathbf{q}\in \mathbf{u}$ is immediate, so we assume $\mathbf{q}\notin \mathbf{u}$.
By the zero-adjunction
criterion (Proposition~\ref{cor:absorption}), this absorption
is valid in $\widehat A_1$ precisely when
\[
 D_\mathbf{q}(\mathbf{u})\ne\varnothing,
 \qquad S(a^2)\models \mathbf{q}\preceq D_\mathbf{q}(\mathbf{u}).
\]
We now describe all such absorptions. In the following
cases, the mentioned summands belong to the restricted support
$D_\mathbf{q}(\mathbf{u})$, and the necessity arguments use the old flat
semiring $S(a^2)=\{a,b,\infty\}$, with $b=a^2$.

\smallskip\noindent
\emph{Case 1: $\mathbf{q}=x$.}
The restricted support can contain only $x,x^2,x^3$.
Validity forces $x$ or $x^3$ to occur. Otherwise the only
possible surviving word is $x^2$; assigning $x=a$ gives the
source value $b$ and the target value $a$, contradicting
absorption. Conversely, $x$ is already absorbed by itself,
and $x^3+x\approx x^3$ handles the other possibility.

\smallskip\noindent
\emph{Case 2: $\mathbf{q}=x^2$.}
Again the restricted support consists of words among
$x,x^2,x^3$. Validity forces $x^2$ or $x^3$ to occur:
otherwise evaluation at $x=a$ separates the source $x$
from the target $b$. Conversely, idempotence handles $x^2$,
and $x^3+x^2\approx x^3$ follows from the already derived
relation $x^2+x^3\approx x^3$.

\smallskip\noindent
\emph{Case 3: $\mathbf{q}=xy$, $x\ne y$.}
The valid cases are precisely the following: $xy$ already
occurs; both $x^2$ and $y^2$ occur; a high-degree normal word
$\mathbf{t}_{D}$ with $\varnothing\ne D\subseteq\{x,y\}$ occurs; or
a linear summand also occurs in one of the quadratic
summands. These cases are derivable, respectively, by
idempotence, \eqref{eq:basis-quadratic},
the high-degree absorption proved above, or first
the linear--quadratic overlap reduction proved above and then
the high-degree absorption proved above.
To prove necessity, assume that none occurs. Then the
restricted support has no high-degree word, has no $xy$,
has at most one of $x^2,y^2$, and has no linear--quadratic
overlap. With no square, the assignment $x=y=a$ makes the
source either $0$ (if the restriction is empty) or $a$, while
the target has value $b$. If $x^2$ is the sole square,
the only possible linear word is $y$; take $x=a$ and $y=b$.
The source has value $b$ and the target $xy$ has value
$\infty$. The argument with $y^2$ is symmetric. Thus the
absorption fails unless one of the stated alternatives holds.

\smallskip\noindent
\emph{Case 4: $\mathbf{q}=\mathbf{t}_{C}$ for nonempty $C$.}
The absorption is valid precisely when the restricted source
$D_\mathbf{q}(\mathbf{u})$ contains a high-degree normal word, or when one of
its linear variables also occurs in a quadratic summand.
For sufficiency, a high-degree word must be $\mathbf{t}_{D}$ with
$\varnothing\ne D\subseteq C$, and
the high-degree absorption proved above absorbs $\mathbf{q}$.
A linear--quadratic overlap first gives such a $\mathbf{t}_{D}$ via
the linear--quadratic overlap reduction proved above.
For necessity, suppose neither condition holds. Let $L$
be the set of variables of linear summands of $D_\mathbf{q}(\mathbf{u})$, and
let $Q$ be the set of variables occurring in its quadratic
summands; these sets are disjoint. Assign $b$ to the
variables in $L$, and $a$ to the variables of $C\setminus L$.
Every surviving source word then has value $b$ (or the
restricted source is empty), while $\mathbf{t}_{C}$ has value $\infty$.
Hence the required absorption fails in $\widehat A_1$.

\smallskip\noindent
\emph{Completion.}
Every term is equivalent, using Lemma~\ref{lem:norm-four},
to a nonempty sum of the four kinds of normal words treated
above. Every valid absorption of one normal word by such
a term is covered by one of the four cases, and each case
shown to be valid has been derived from the six identities.
Finally, let $\mathbf{u}\approx \mathbf{v}$ be an arbitrary identity of
$\widehat A_1$. Normalize the word summands of both sides.
For each word $\mathbf{p}$ of $\mathbf{u}$, derive $\mathbf{p}\preceq \mathbf{v}$; for each
word $\mathbf{q}$ of $\mathbf{v}$, derive $\mathbf{q}\preceq \mathbf{u}$.
Adding the finitely many absorbed words one at a time then
gives
\[
 \mathbf{u}\approx \mathbf{u}+\mathbf{v}\approx \mathbf{v}.
\]
Thus every identity of $\widehat A_1$ follows from the six
displayed identities and the ai-semiring axioms. This proves
that they form an equational basis of $\W_1$.
\end{proof}

\subsection{The position of \texorpdfstring{$SR_6$}{SR6} in \texorpdfstring{$\W_1$}{W1}}

\begin{proposition}\label{lem:sr6-hsp}
The variety $\Var(SR_6)$ is a subvariety of $\W_1$.
\end{proposition}
\begin{proof}
Shao--Gao~\cite[Proposition~25.3]{ShaoGao26} prove that
$\Var(SR_6)\subseteq\Var(S(a^2)^0)$.
Since $\widehat A_1\cong S_c(a^2)^0\cong S(a^2)^0$,
the asserted inclusion follows.
\end{proof}

\subsection{The limit subvariety and the hereditary finite-basis reduction}

By Lyu--Ren--Yue~\cite{LyuRenYue26}, $\Var(SR_6)$ is a limit
variety. Yue~\cite[Corollary~3.35]{Yue26} identifies $\Var(SR_6)$ and
$\Var(S_c(abc))$ as the only limit subvarieties of
$\Var(S_7^0)$. The arguments below use the valid identities of
Proposition~\ref{lem:six-valid} and the resulting normal forms in
Lemma~\ref{lem:norm-four}.

We use the semiring-inequality notation introduced in
Section~\ref{sec:prelim}. A variety is \emph{hereditarily
finitely based} if all its subvarieties are finitely based.
The following theorem yields the asserted uniqueness.

\begin{theorem}\label{thm:unique-limit}
Let $\mathcal K\subseteq\W_1$ be a subvariety such that
$\Var(SR_6)\nsubseteq\mathcal K$. Then $\mathcal K$ is
hereditarily finitely based. Consequently, $\Var(SR_6)$ is
the unique limit subvariety of $\W_1$.
\end{theorem}

The reduction will be formulated using four additional identities:
\begin{align*}
  \lambda &: x+xy\approx xy, \\
  \sigma &: xy\approx x^2+y^2, \\
  \rho &: xy\approx x^2y, \\
  \eta &: x_1x_2+x_2x_3+x_3x_4
     \approx x_1x_2+x_2x_3+x_3x_4+x_1x_4.
\end{align*}
Denote by $\W_1[\varepsilon]$ the subvariety of $\W_1$
defined by an additional identity $\varepsilon$.

\begin{proposition}\label{lem:limit-small-slices}
The varieties $\W_1[\lambda]$, $\W_1[\sigma]$, and
$\W_1[\rho]$ are hereditarily finitely based. More precisely,
\[
  \W_1[\lambda]=\Var(E_{\mathrm{nil}}),\qquad
  \W_1[\sigma]=\V_1,\qquad
  \W_1[\rho]=\Var(R_{\mathrm{ch}}),
\]
where $E_{\mathrm{nil}}=\{e,\infty\}$ has $e+\infty=\infty$ and all products
$\infty$, and $R_{\mathrm{ch}}=\{0,e,\infty\}$ is the chain $0<e<\infty$
with $0$ multiplicatively absorbing and all products of elements
of $\{e,\infty\}$ equal to $\infty$.
\end{proposition}
\begin{proof}
\emph{1. The $\lambda$-branch.}
Suppose first that $\lambda$ holds in $\W_1$. Substituting $y=x$ into
$\lambda$ and using \eqref{eq:basis-square} gives $x^2\approx x^3$.
Multiplication of $\lambda$ by $x$ and \eqref{eq:basis-absorb} now yield
\[
  x^2+x^2y\approx x^2y,
  \qquad x^2+x^2y\approx x^2;
\]
thus $x^2y\approx x^2$. By \eqref{eq:derived-square} and commutativity,
\[
             x^2\approx x^2y\approx xy^2\approx y^2.
\]
Denote this common square term by $\mathbf{c}$. Equations \eqref{eq:basis-square}--\eqref{eq:basis-absorb}
then give $x+\mathbf{c}\approx \mathbf{c}$, $\mathbf{c}x\approx \mathbf{c}$, and
$x+xy\approx \mathbf{c}$. Using $\lambda$ we obtain $xy\approx \mathbf{c}$.
Consequently the variety is that of join semilattices with a
greatest element which is the value of every product. Such a
semilattice is separated by homomorphisms into $E_{\mathrm{nil}}$: for
$a\not\le b$, the map taking elements below $b$ to $e$ and the
remaining elements to $\infty$ is a semilattice homomorphism
(and preserves multiplication). Hence $\W_1[\lambda]=\Var(E_{\mathrm{nil}})$.
Every nontrivial subvariety contains $E_{\mathrm{nil}}$, so this variety has
only two subvarieties.

\medskip\noindent\emph{2. The $\sigma$-branch.}
Equation~\eqref{eq:basis-quadratic}
shows that $\sigma$ is equivalent over $\W_1$
to $x^2+xy\approx xy$. In $\W_1[\sigma]$,
\[
  x^3+y^3\approx x^2y^2\approx x^3y\preceq x^3.
\]
Here the first equality follows by applying $\sigma$ to $x^2,y^2$
and using \eqref{eq:derived-power}, the second from \eqref{eq:derived-square}, and the absorption from
\eqref{eq:basis-absorb}, with $y$ replaced by $x^2y$. Interchanging $x,y$ gives $x^3\approx y^3$.
Let $\mathbf{c}$ denote this common term. By \eqref{eq:basis-square}, $x+\mathbf{c}\approx \mathbf{c}$;
\eqref{eq:basis-absorb} makes $\mathbf{c}$ absorb every quadratic term. Every term is therefore reducible either to $\mathbf{c}$ or to
\begin{equation}\label{eq:limit-V1-normal}
\begin{aligned}
 \mathbf{F}(L,Q)&=\mathbf{p}_1+\cdots+\mathbf{p}_r,\\
 \{\mathbf{p}_1,\ldots,\mathbf{p}_r\}&=L\cup\{y^2:y\in Q\},\\
 L\cap Q&=\varnothing,\qquad L\cup Q\ne\varnothing.
 \end{aligned}
\end{equation}
Indeed, $\sigma$ removes mixed quadratic terms;
\eqref{eq:basis-comm} and \eqref{eq:basis-cubic}, together with
\eqref{eq:derived-power} and \eqref{eq:derived-square},
reduce higher-degree terms to $\mathbf{c}$, and \eqref{eq:basis-square} removes a linear
term accompanied by its square. For separation, put $b=a^2$. For each pair $(L,Q)$, define the valuation
\[
 \nu_{L,Q}(x)=
 \begin{cases}
 b,&x\in L,\\
 a,&x\in Q,\\
 \infty,&x\notin L\cup Q.
 \end{cases}
\]
Then $\mathbf{F}(L,Q)^{\nu_{L,Q}}=b$, whereas $\mathbf{c}^{\nu_{L,Q}}=\infty$.
Moreover, for a nonempty $\mathbf{F}(L',Q')$ to have value $b$ under this
valuation, each variable of $L'$ must lie in $L$ and each variable
of $Q'$ must lie in $Q$. Indeed, a linear summand has value $b$
only at $b$, and a squared summand has value $b$ only at
$a$; any other value of a summand makes the flat sum
different from $b$. Consequently, if neither $(L',Q')$ nor
$(L,Q)$ is coordinatewise contained in the other, one of the two
associated valuations distinguishes the forms. If
$(L',Q')\ne(L,Q)$ and, say, $L'\subseteq L$, $Q'\subseteq Q$,
then $\nu_{L',Q'}$ gives $b$ on $\mathbf{F}(L',Q')$ and not on $\mathbf{F}(L,Q)$.
The latter contains a summand with value $\infty$, from a variable
in $(L\setminus L')\cup(Q\setminus Q')$. This also covers changes from linear to quadratic status.
Thus the forms, including $\mathbf{c}$, are pairwise distinct in $A_1$.
For clarity, both inclusions are needed here. The old
subsemiring $A_1=\{a,b,\infty\}$ of $\widehat A_1$
satisfies $\sigma$ by its multiplication and flat-addition
tables. Hence $\Var(A_1)\subseteq\W_1[\sigma]$.
Conversely, the preceding normal-form reduction applies in
$\W_1[\sigma]$, and $A_1$ separates any two distinct
normal forms. Therefore every identity holding in $A_1$
holds in $\W_1[\sigma]$, proving
$\W_1[\sigma]\subseteq\Var(A_1)$.

For hereditary finite basability, reduce identities to
absorptions $\mathbf{q}\preceq \mathbf{u}$, with $\mathbf{u}=\mathbf{F}(L,Q)$ and,
after decomposing a quadratic query under $\sigma$,
$\mathbf{q}\in\{x,x^2,\mathbf{c}\}$ up to the identities of $\V_1$.
Mark the query variable when $\mathbf{q}=x$ or $x^2$; for $\mathbf{c}$,
use the cube of a fresh query variable. The finite state
records the query type and whether its mark lies in $L$, in
$Q$, or outside $L\cup Q$. Let $M$ be the set of marks and put
\[
 \mathbf d(\mathbf{u},\mathbf{q})=
 \bigl(|L\setminus M|,|Q\setminus M|\bigr)\in\mathbb N^2.
\]
For equal states, compare the vectors coordinatewise.
By Dickson's lemma this is a well-quasi-order. If
$\mathbf d(\mathbf{u},\mathbf{q})\le\mathbf d(\mathbf{u}',\mathbf{q}')$, map the marked
variable to its corresponding mark, and inject the remaining
$L$- and $Q$-variables separately into their larger supports.
Map a fresh cube variable to a fresh cube variable when
$\mathbf{q}=\mathbf{c}$. This single renaming satisfies
$\theta(\mathbf{q})\approx \mathbf{q}'$ and $\mathbf{u}'\approx\theta(\mathbf{u})+\mathbf{r}$,
so the absorption derivation is
\[
 \mathbf{u}'+\mathbf{q}'\approx\theta(\mathbf{u})+\mathbf{r}+\theta(\mathbf{q})\approx \mathbf{u}'.
\]
Here the last step uses the smaller absorption under
$\theta$; if there are no extra summands, omit $\mathbf{r}$. Therefore the valid absorptions
of any subvariety have finitely many minimal generators
among the finitely many states. Together with $\sigma$
and a finite basis for $\W_1$~\cite{ZhaoWu25}, these
form a finite basis for that subvariety.

\medskip\noindent\emph{3. The $\rho$-branch.}
Substitution of $y=x$ gives
$x^2\approx x^3$. By $\rho$, \eqref{eq:derived-square} and \eqref{eq:basis-cubic}, any word
of degree at least two depends only on its content: for
$\varnothing\ne Y$, represent it by
\[
 \mathbf{u}_Y=\begin{cases}x^2,&Y=\{x\},\\
                    \prod_{x\in Y}x,&|Y|\ge2.
       \end{cases}
\]
Whenever $Y\subseteq Z$, \eqref{eq:basis-absorb} and $\rho$ imply
$\mathbf{u}_Y+\mathbf{u}_Z\approx \mathbf{u}_Y$. Indeed, $\mathbf{u}_Y^2\approx \mathbf{u}_Y$, so
\eqref{eq:basis-absorb} with $x=\mathbf{u}_Y$ yields
$\mathbf{u}_Y+\mathbf{u}_Yz\approx \mathbf{u}_Y$. If $|Y|\ge2$, the content
reduction identifies $\mathbf{u}_Z$ with $\mathbf{u}_Y$ multiplied by the
variables in $Z\setminus Y$. If $Y=\{x\}$ and
$Z\supsetneq Y$, the $\rho$-instance
$xz\approx x^2z$, where $z$ is the product of the
variables of $Z\setminus\{x\}$, gives the same
identification with $\mathbf{u}_Y=x^2$. The case $Z=Y$ follows
by additive idempotence. Moreover \eqref{eq:basis-square} gives
$x+x^2\approx x^2$. Hence each term has the finite-word form
\[
 \mathbf{F}(L,\mathcal A)=\mathbf{p}_1+\cdots+\mathbf{p}_r,\qquad
 \{\mathbf{p}_1,\ldots,\mathbf{p}_r\}=L\cup\{\mathbf{u}_Y:Y\in\mathcal A\},
\]
where $\mathcal A$ is an antichain of nonempty subsets of
variables, $x\notin L$ if $\{x\}\in\mathcal A$, and the
support is nonempty. Distinct forms are separated in $R_{\mathrm{ch}}$ as follows. For each nonempty
finite set $Y$ of variables, let $\nu_Y$ take value $e$ on $Y$
and $0$ elsewhere. For this evaluation, $\mathbf{u}_Z^{\nu_Y}=\infty$
if and only if $Z\subseteq Y$, while every linear
summand takes value at most $e$.
Suppose $\mathcal A\ne\mathcal A'$. Choose an inclusion-minimal set $Y$ that belongs to
exactly one of $\mathcal A$ and $\mathcal A'$, and, after
interchanging the two forms if necessary, let
$Y\in\mathcal A\setminus\mathcal A'$. No member
$Z\in\mathcal A'$ can be contained in $Y$: equality
is excluded, and a proper subset would also belong to exactly one of the two antichains,
contrary to the choice of $Y$. Under
$\nu_Y$, the first form therefore takes value $\infty$
while the second takes value at most $e$. This works
regardless of their linear supports.
If $\mathcal A=\mathcal A'$ but $L\ne L'$, choose $x$ belonging to exactly one of $L,L'$, and assume
$x\in L\setminus L'$.
The normalization condition implies $\{x\}\notin\mathcal A$.
Under $\nu_{\{x\}}$, every quadratic summand of either
form is $0$, and only the first form contains the linear
summand $x=e$. Hence the two values are $e$ and $0$, respectively.
The normal forms are therefore distinct in $R_{\mathrm{ch}}$.
The subset $\{0,b,\infty\}$ of $\widehat A_1$
is a subsemiring isomorphic to $R_{\mathrm{ch}}$
(the product of any two elements in $\{b,\infty\}$
is $\infty$). Its table satisfies $\rho$, so
$\Var(R_{\mathrm{ch}})\subseteq\W_1[\rho]$.
The preceding normal-form reduction and separation show
that every identity of $R_{\mathrm{ch}}$ holds in
$\W_1[\rho]$. Hence
$\W_1[\rho]\subseteq\Var(R_{\mathrm{ch}})$,
and the required equality follows.

The three-element semiring $R_{\mathrm{ch}}$ is isomorphic
to the zero extension $S(a)^0$: its elements are $0,a,\infty$,
its addition is $0<a<\infty$, and all products of the two
old elements $a,\infty$ equal $\infty$. Shao--Gao
\cite[Theorem~25.2]{ShaoGao26} prove that $\Var(S(a)^0)$
has precisely five subvarieties and that each is finitely
based. Since we established $\W_1[\rho]=\Var(R_{\mathrm{ch}})$
above, this proves the hereditary finite-basis assertion
directly. Note that Shao--Gao denote $\Var(S(a)^0)$ by
$\mathcal V_1$, whereas our $\V_1=\Var(S(a^2))$;
these two notations refer to different varieties.
\end{proof}

We treat $\W_1[\eta]$ by means of graph normal forms
and a well-quasi-order on their finite encodings.

\begin{proposition}\label{lem:limit-eta-HFB}
Every subvariety of $\W_1[\eta]$ is finitely based.
\end{proposition}
\begin{proof}
\emph{1. Collapse of all words of length at least three.}
We first show that all cubes agree in $\W_1[\eta]$.
Put $a=x^3$, $b=y^3$ and substitute $(a,b,b,a)$
into $\eta$. Since $ab+b^2\approx b$ by
\eqref{eq:derived-power} and \eqref{eq:basis-absorb},
this instance reduces to $b+a^2\approx b$.
As $a^2\approx a$ by \eqref{eq:derived-power}, we obtain
$b+a\approx b$. Interchanging $x$ and $y$ gives
$a+b\approx a$. By commutativity of addition,
$x^3\approx y^3$. Denote the common cube value by $\mathbf{c}$.
The derived identities now give
\[
 \begin{aligned}
 x^3y&\approx y^3y\approx y^4\approx y^3=\mathbf{c},\\
 x+\mathbf{c}&\approx x+x^2+x^3\approx x^3=\mathbf{c},\\
 \mathbf{c}+xy&\approx \mathbf{c}.
 \end{aligned}
\]
Here the first line uses the common cube and
\eqref{eq:derived-power}, the second uses
\eqref{eq:basis-square}, and the third uses
\eqref{eq:basis-absorb}. Finally, \eqref{eq:basis-mixed} yields the explicit deduction
\[
\begin{aligned}
 x+xy
 &\approx x+xy+x^3y &&\text{by \eqref{eq:basis-mixed}}\\
 &\approx x+xy+\mathbf{c} &&\text{since }x^3y\approx \mathbf{c}\\
 &\approx x+\mathbf{c} &&\text{since }\mathbf{c}+xy\approx \mathbf{c}\\
 &\approx \mathbf{c} &&\text{since }x+\mathbf{c}\approx \mathbf{c}.
\end{aligned}
\]
Consequently,
\begin{equation}\label{eq:limit-cube-conseq}
       x^3\approx \mathbf{c},\qquad \mathbf{c}x\approx \mathbf{c},\qquad
       x+\mathbf{c}\approx \mathbf{c},\qquad x+xy\approx \mathbf{c}.
\end{equation}
In particular every word of length at least three represents
$\mathbf{c}$, and a linear term occurring also in a quadratic term
makes the entire sum equal to $\mathbf{c}$.

\medskip\noindent\emph{2. Graph normal forms.}
We next describe the remaining terms by finite graphs. By
Lemma~\ref{lem:norm-four} and \eqref{eq:limit-cube-conseq}, a term other than
$\mathbf{c}$ is a sum of distinct linear words indexed by a set $L$
and quadratic words indexed by the edges (loops permitted)
of a finite graph $G$, with $L\cap V(G)=\varnothing$.
The identity $\eta$ adds an edge across every three-edge
walk. For an odd walk $v_0,\ldots,v_{2m+1}$, induct on $m$:
\[
 \begin{aligned}
 m&=0:&&v_0v_1\in E(G),\\
 m&=1:&&v_0v_3\text{ is added by }\eta,\\
   m&\ge2:&&v_0v_{2m-1},\ 
 v_{2m-1}v_{2m},\ v_{2m}v_{2m+1},\\
 &&&v_0v_{2m+1}\quad\text{by }\eta.
 \end{aligned}
\]
The induction hypothesis supplies the first edge in the
last line. Each induction step applies $\eta$ in the
additive context of the unchanged summands.
Consequently each bipartite
connected component of $G$ may be replaced by a complete
bipartite graph, while each nonbipartite component becomes a
complete graph with a loop at every vertex. If two looped
components occur, choose a loop $x^2$ in the first and a loop
$y^2$ in the second. The instance
$\,x^2+y^2\approx x^2+y^2+xy\,$ of
\eqref{eq:basis-quadratic} joins them by an edge; odd-walk
closure then combines them into a single complete looped
component. Repeating the argument merges all such components.
Thus every term is equivalent either to $\mathbf{c}$ or to
\begin{equation}\label{eq:limit-graph-normal}
 \mathbf{u}=\mathbf{p}_1+\cdots+\mathbf{p}_r,\qquad
 P(L,G)=\{\mathbf{p}_1,\ldots,\mathbf{p}_r\}=L\cup\{xy:xy\in E(G)\},
\end{equation}
where each undirected edge is represented by either order of
its endpoints (the result is independent of this choice by
\eqref{eq:basis-comm}), and where the linear vertices are disjoint from $G$, the nonlooped
components are complete bipartite graphs with both shores
nonempty, and there is at most one complete looped component.
The sum in \eqref{eq:limit-graph-normal} is nonempty. We do
not assert that these representatives are unique; the subsequent
argument uses their existence and the fact that the displayed
operations are obtained by identities valid in $\W_1[\eta]$.

\medskip\noindent\emph{3. Finite encoding of absorption identities.}
Let $\mathcal K\subseteq\W_1[\eta]$
and consider its absorption identities. In an ai-semiring,
for nonempty sums $\mathbf{u}=\mathbf{p}_1+\cdots+\mathbf{p}_r$ and
$\mathbf{v}=\mathbf{q}_1+\cdots+\mathbf{q}_s$, the identity $\mathbf{u}\approx \mathbf{v}$ holds
if and only if all the absorptions
\[
 \mathbf{p}_i\preceq \mathbf{v}\quad(1\le i\le r),\qquad
 \mathbf{q}_j\preceq \mathbf{u}\quad(1\le j\le s)
\]
hold. Indeed, the first family gives $\mathbf{u}+\mathbf{v}\approx \mathbf{v}$,
whereas the second gives $\mathbf{u}+\mathbf{v}\approx \mathbf{u}$.
Thus, modulo the six identities in
Proposition~\ref{lem:six-valid} and $\eta$, it suffices
to consider absorptions $\mathbf{q}\preceq \mathbf{u}$ with $\mathbf{q}$ a word
and $\mathbf{u}$ a nonempty sum of words. Absorptions with source
$\mathbf{u}=\mathbf{c}$ are automatic
by \eqref{eq:limit-cube-conseq} and may be discarded.
For the remaining absorptions, $\mathbf{u}$ has the form
\eqref{eq:limit-graph-normal} and $\mathbf{q}$ is either $\mathbf{c}$, a
linear word, or a quadratic word (a loop is allowed).
In the latter two cases mark the at most two variables
of $\mathbf{q}$. We shall use the following comparison of the
resulting pairs $(\mathbf{u},\mathbf{q})$.

The marked configurations are classified by the form of the query
and the positions of its variables. After normalization, the query
is either $\mathbf{c}$, $x$, $x^2$, or $xy$ with $x\ne y$. Each marked
variable is either absent from the source, occurs in its linear
part, or belongs to a bipartite or looped component of the graph.
For $xy$, we also record whether its variables belong to the same
graph component; if they lie in the same bipartite component, we
distinguish the same shore from opposite shores. A square has one
marked variable, whereas $xy$ has two distinct marks. Since the
linear part is disjoint from the graph and there is at most one
looped component, these possibilities exhaust all configurations.
The presence or absence of the looped component is also recorded
in the finite state.
Within a fixed state, distinguish the at most two bipartite
components containing marked vertices, and orient their shores
consistently with the marks. The state also records how many
of the two query marks occupy each distinguished component,
so the sizes below count only \emph{unmarked} vertices.
For an unmarked bipartite component choose either shore
orientation; these components can be listed in any fixed order.
The source is encoded by a vector of nonnegative integers
consisting of the unmarked linear vertices, the unmarked
vertices of the looped component (if present), and the two
unmarked shore sizes in each distinguished bipartite component,
together with a finite sequence
\[
              (s_1,t_1),\ldots,(s_j,t_j)
                  \qquad(s_i,t_i\ge1)
\]
for the unmarked bipartite components. Whether the looped
component is present is included in the finite state.

To avoid comparing arbitrarily chosen representations, fix
\emph{one encoding for each pair} $(\mathbf{u},\mathbf{q})$. Orient the shores of
each unmarked component so that the smaller shore comes first,
and list these components in lexicographic order of the resulting
size pairs. Resolve ties by a fixed convention. Represent a chosen
encoding as
\[
 e=(\tau;\mathbf d;(s_1,t_1),\ldots,(s_j,t_j)),
\]
where $\tau$ is the finite state and $\mathbf d$ the
finite vector just described. Declare $e\le e'$ precisely
when their states agree, $\mathbf d\le\mathbf d'$ coordinatewise,
and some indices $1\le i_1<\cdots<i_j\le j'$ satisfy
\[
 (s_k,t_k)\le(s'_{i_k},t'_{i_k})
 \qquad(1\le k\le j)
\]
coordinatewise. Dickson's lemma for the vectors and
$\mathbb N^2$, followed by Higman's lemma~\cite{Higman52},
implies that every infinite sequence of fixed encodings has
an increasing pair (there are only finitely many states).
The empty list encodes zero unmarked components; absence
of the unique looped component belongs to $\tau$.
Neither denotes an empty polynomial.

\medskip\noindent\emph{4. A single substitution for each comparison.}
We realize every encoding comparison by one substitution and
an equational derivation.
Suppose $(\mathbf{u},\mathbf{q})$ has an encoding below that of $(\mathbf{u}',\mathbf{q}')$. Match the
(at most two) marked variables of $\mathbf{q}$ to the corresponding
variables of $\mathbf{q}'$, preserving whether they coincide and which
part of the source each occupies. Send the unmarked linear
variables injectively into the corresponding linear variables
of $\mathbf{u}'$; the coordinate inequality guarantees room for them.
Do the same for the unique looped component, if present, and
for each of the at most two marked bipartite components,
embedding the two shores separately while fixing all marked
vertices. For an unmarked bipartite component, the Higman
subsequence matching selects a distinct target component;
possibly reverse its shores and inject each of the two vertex
sets into the corresponding target shore. Map every query
variable outside $\mathbf{u}$ to its designated query variable outside
$\mathbf{u}'$. The choices are made simultaneously and define a single
injective renaming $\theta$. In particular, every occurrence
of a source variable receives the same image; choosing separate
images for individual graph edges would not define a substitution.

The construction handles every source variable according to
one of four disjoint cases:
\begin{enumerate}
\item A marked variable is sent to the matching marked
variable, preserving equality or distinctness of the marks.
\item An unmarked linear variable is sent injectively to an
unmarked linear variable; a mark in the linear part has
already been fixed in the first case.
\item An unmarked vertex of the looped component or of a
marked bipartite component is sent injectively within the
corresponding target component and shore, avoiding the
previously fixed marks.
\item A vertex of an unmarked bipartite component is sent
into the distinct target component selected by the Higman
subsequence match, preserving its shore after the permitted
orientation reversal.
\end{enumerate}
Variables occurring only in $\mathbf{q}$ belong to the first case;
when $\mathbf{q}=\mathbf{c}$ a fresh cubic variable is fixed separately.
The recorded finite state fixes the positions of all marks,
so the injections in the other three cases have disjoint
images. Completeness of the bipartite and looped components
ensures that every source edge maps to a target edge;
linear summands map to linear summands. Consequently
$\theta$ sends every summand of $\mathbf{u}$ to a summand of $\mathbf{u}'$.
This is a single injective renaming, not a family of
independent edgewise assignments.
If $\mathbf{q}$ is linear or quadratic, $\theta(\mathbf{q})=\mathbf{q}'$ by the marks;
if $\mathbf{q}=\mathbf{c}$, represent it by $z^3$ with a fresh $z$ and set
$\theta(z)=z'$, using the equality of all cubes. Hence,
modulo the identities of $\W_1[\eta]$,
\[
 \mathbf{q}'\approx\theta(\mathbf{q}),\qquad
 \mathbf{u}'\approx\theta(\mathbf{u})+\mathbf{r},
\]
where $\mathbf{r}$ consists of the extra summands of $\mathbf{u}'$. When
none remain, the second relation reads $\mathbf{u}'\approx\theta(\mathbf{u})$
with no empty-sum term.

For completeness, the actual equational deduction is
\[
 \begin{aligned}
 \mathbf{u}'+\mathbf{q}'&\approx\theta(\mathbf{u})+\mathbf{r}+\theta(\mathbf{q})\\
      &\approx\theta(\mathbf{u})+\theta(\mathbf{q})+\mathbf{r}\\
      &\approx\theta(\mathbf{u})+\mathbf{r}\approx \mathbf{u}'.
 \end{aligned}
\]
The middle step is the $\theta$-instance of the assumed
absorption $\mathbf{u}+\mathbf{q}\approx \mathbf{u}$, used inside an additive context.
If $\mathbf{r}$ is absent, the same calculation is read without it.

\medskip\noindent\emph{5. Finite basis.}
The displayed derivation shows that the absorptions valid
in $\mathcal K$ form an upward-closed set in the encoding order.
This order is a well-quasi-order by Dickson's and Higman's
lemmas; hence the set has finitely many minimal generators. These absorptions, together with $\eta$ and a finite
basis for $\W_1$ (available by~\cite{ZhaoWu25}), form a
finite equational basis for $\mathcal K$.
\end{proof}

For the remaining reduction we record the evaluations in $SR_6$. We label its elements $1,\ldots,6$ as in
\cite{LyuRenYue26}. The element
$1$ is the additive top and is multiplicatively absorbing;
all squares and all products of length at least three are $1$.
The products distinct from $1$ are
\begin{equation*}
       2\cdot6=5\cdot6=5\cdot4=3
\end{equation*}
and their symmetric counterparts. The element $3$ has
$3+z=1$ for every $z\ne3$. Direct evaluation in the Cayley tables shows that the substitution
$(x_1,x_2,x_3,x_4)=(2,6,5,4)$ gives the left side of
$\eta$ value $3$ and its right side value $1$; hence
$SR_6$ does not satisfy $\eta$.

\begin{lemma}\label{lem:limit-fold}
Let $\mathcal K\subseteq\W_1$ and suppose
$\Var(SR_6)\nsubseteq\mathcal K$. Then $\mathcal K$
is contained in one of
\[
  \W_1[\lambda],\qquad \W_1[\sigma],\qquad
  \W_1[\rho],\qquad \W_1[\eta].
\]
\end{lemma}
\begin{proof}
\emph{1. A failing absorption and its bipartite source.}
Choose an identity holding in $\mathcal K$ but failing in $SR_6$.
By distributivity and additive idempotence, one of its
component absorptions $\mathbf{q}\preceq \mathbf{u}$ is valid in $\mathcal K$
but fails in $SR_6$; here $\mathbf{q}$ is a word and $\mathbf{u}$ a nonempty
sum of words. All words may be taken commutative. Since
$SR_6$ has a valuation at which $\mathbf{u}\ne1$, its multiplication
table shows that $\mathbf{u}$ contains only linear words and
square-free quadratic words. Write
\[
 \mathbf{u}=\mathbf{p}_1+\cdots+\mathbf{p}_r,\quad \{\mathbf{p}_1,\ldots,\mathbf{p}_r\}=P(L,G),\qquad
 P(L,G)=L\cup\{xy:xy\in E(G)\},
\]
with each edge word represented in either order, as permitted
by \eqref{eq:basis-comm}.
The graph $G$ is bipartite and $L\cap V(G)=\varnothing$:
indeed every quadratic summand must have value $3$, so the
endpoints of its edges must be assigned elements in the two
classes $\{2,5\}$ and $\{4,6\}$; an odd cycle is impossible,
and a variable occurring also as a linear summand cannot
simultaneously have value $3$. Conversely, such a sum has a
valuation of value $3$, by assigning $5,6$ on the two shores
and $3$ to the linear variables. The case $E(G)=\varnothing$
is included.

\medskip\noindent\emph{2. Reduction to a two-letter source.}
For each bipartite component choose its shores $U,V$, and
use the simultaneous substitution
\[
 \theta(z)=
 \begin{cases}
 s,&z\in U,\\
 t,&z\in V,\\
 st,&z\in L.
 \end{cases}
\]
For this simultaneous substitution, $\theta(\mathbf{u})\approx st$.
Variables outside $c(\mathbf{u})$ may be sent to $s,t$, or $st$
as specified in the cases below. Each component is oriented
independently, but every occurrence of a variable receives
the same image.

\medskip\noindent\emph{3. Linear, square, and higher-degree queries.}
If $\mathbf{q}$ has degree at least three, or is a square, the
image $\theta(\mathbf{q})$ is a word of degree at least two and is not $st$.
By commutativity, \eqref{eq:derived-power} and
\eqref{eq:derived-square}, its possibilities modulo the identities
of $\W_1$ reduce to
\[
 s^2,\quad t^2,\quad s^3,\quad t^3,\quad s^2t
 \quad(\text{with }s^2t\approx st^2).
\]
One-variable powers of degree at least three become cubes,
and mixed words of degree at least three become $s^2t$
by \eqref{eq:derived-power} and \eqref{eq:derived-square}.
Because $\mathcal K\models\theta(\mathbf{q})\preceq st$, the
possibilities yield the following consequences,
with their hypotheses stated explicitly:
\[
\begin{array}{ll}
 \text{hypothesis }st+s^2\approx st:&
 st\approx st+s^2+t^2\approx s^2+t^2\quad(\sigma),\\[3pt]
 \text{hypothesis }st+s^3\approx st:&
 st+s+s^2\approx st+s^3\approx st\quad(\lambda),\\[3pt]
 \text{hypothesis }st+s^2t\approx st:&
 st\approx s^2t\quad(\rho).
\end{array}
\]
For the first line interchange $s$ and $t$ and use
\eqref{eq:basis-quadratic}; for the second use
\eqref{eq:basis-square}; for the third use
$st+s^2t\approx s^2t$, obtained by multiplying
\eqref{eq:basis-square} by $t$ and applying
\eqref{eq:six-square-cube}. The $t^2,t^3$ cases are
symmetric. If $\mathbf{q}$ is linear, it cannot belong to $L$
(because $\mathbf{q}\preceq \mathbf{u}$ would then hold in $SR_6$).
Orient its graph component, or assign it if absent from
$c(\mathbf{u})$, so that $\theta(\mathbf{q})=s$. Thus
$\mathcal K\models s\preceq st$, giving $\lambda$.

\medskip\noindent\emph{4. Square-free quadratic queries.}
Let $\mathbf{q}=xy$ with $x\ne y$. If either
endpoint belongs to $L$ or is absent from $\mathbf{u}$, map that
endpoint to $st$ and orient the remaining components as
above. The image of $\mathbf{q}$ has degree at least three. If
$x,y$ lie in different components of $G$, or on the same
shore of one component, the orientations can instead be
chosen so that both map to $s$. We obtain $st\succeq s^2$
and hence $\sigma$.

In the remaining case $x,y$ lie in opposite shores of the
same connected component. Since the absorption fails in
$SR_6$, the word $xy$ cannot be a summand of $\mathbf{u}$.
Choose a shortest odd path $x=v_0,\ldots,v_{2m+1}=y$
with $m\ge1$. Let $C_0$ be their component, with shores
$U\ni x$ and $V\ni y$. For each other component $C$, choose
shores $U_C,V_C$. Define the single substitution
\[
\theta(z)=
\begin{cases}
 x_1,&z=x,\\
 x_4,&z=y,\\
 x_3,&z\in (U\setminus\{x\})\cup\bigcup_{C\ne C_0}U_C,\\
 x_2,&z\in (V\setminus\{y\})\cup\bigcup_{C\ne C_0}V_C,\\
 x_2x_3,&z\in L.
\end{cases}
\]
The cases are disjoint and assign one word to each source
variable. Thus
\[
 \theta(\mathbf{u})\approx x_1x_2+x_2x_3+x_3x_4,
 \qquad \theta(\mathbf{q})=x_1x_4.
\]
Indeed the first and last edges of the chosen path produce
$x_1x_2$ and $x_3x_4$; an internal edge produces $x_2x_3$.
No edge produces $x_1x_4$, since $xy\notin E(G)$.
Here no additional edge image is possible: an edge incident
with $x$ maps to $x_1x_2$, one incident with $y$ maps to
$x_3x_4$, and every other edge within this component maps
to $x_2x_3$. All other components and all linear summands
also map to $x_2x_3$. The shortest path has length at least
three and supplies all three displayed source terms. The image
of the query is $x_1x_4$, so $\mathcal K$ satisfies $\eta$.
The cases exhaust every query word: linear, square or
higher-degree, or square-free quadratic; for the last type,
the endpoints are outside/in the linear part of the source,
in distinct components, on the same shore, or on opposite
shores of one component. Each construction uses one
substitution for \emph{all} source summands. Hence one of
$\lambda,\sigma,\rho,\eta$ holds in $\mathcal K$.
The final odd-path substitution also appears in
Yue\,\cite[Proposition~3.21]{Yue26}.
\end{proof}

The hereditary finite-basis reduction now yields
Theorem~\ref{thm:unique-limit}.
\begin{proof}
Proposition~\ref{lem:sr6-hsp} gives $\Var(SR_6)\subseteq\W_1$
without using a defining identity basis for $\W_1$.
The variety $\Var(SR_6)$ is a limit variety
by~\cite{LyuRenYue26}.

Let $\mathcal K\subseteq\W_1$ with
$\Var(SR_6)\nsubseteq\mathcal K$. By
Lemma~\ref{lem:limit-fold}, $\mathcal K$ lies in one of the
four subvarieties listed there. Each of these is
hereditarily finitely based by
Propositions~\ref{lem:limit-small-slices} and
\ref{lem:limit-eta-HFB}.
This proves the first assertion.

Suppose that $\mathcal K$ is a limit subvariety of $\W_1$
distinct from $\Var(SR_6)$. Then either
$\Var(SR_6)\nsubseteq\mathcal K$, contradicting the first
assertion, or $\Var(SR_6)\subsetneq\mathcal K$,
contradicting the definition of a limit variety.
Therefore $\Var(SR_6)$ is the unique limit subvariety of
$\W_1$.
\end{proof}

\section{A limit variety generated by \texorpdfstring{$T_\infty$}{T-infinity}}\label{sec:path-limit-section}

The limit property of $\Var(SR_6)$ is due to
Lyu--Ren--Yue~\cite{LyuRenYue26}; the uniqueness statement
proved in Section~\ref{sec:unique-limit} is relative to
$\W_1$. We construct a further limit variety, distinct from
$\Var(SR_6)$,
which lies in every $\W_n$ for $n\ge2$ but in no $\V_n$.
The proof has three parts: nonabsorbing triple products encode
directed edges; finite directed cycles provide countermodels to
finite basability; and a bound on directed-path length controls
proper subvarieties. The embedding and separating identities
are verified before the structural argument.

\subsection{Semirings attached to directed paths and their location in \texorpdfstring{$\W_2$}{W2}}\label{sec:path-construct}

For $k\ge2$, let
\[
 T_k=\{\infty,\omega,e\}\mathbin{\dot\cup}
      \{u_i,p_i,q_i:1\le i\le k\}.
\]
Addition is flat: $a+a=a$ and $a+b=\infty$ for distinct elements.
The products different from $\infty$ are precisely
\begin{equation*}
\begin{aligned}
 u_i e&=p_i,& e u_i&=q_i &&(1\le i\le k),\\
 p_i u_{i+1}&=\omega,&u_i q_{i+1}&=\omega &&(1\le i<k).
\end{aligned}
\end{equation*}
All remaining products, including products involving $\omega$ or
$\infty$, equal $\infty$. Define $T_\infty$ by the same rules with indices
in the positive integers and put
\begin{equation*}
 \mathcal U=\Var(T_\infty).
\end{equation*}
We retain the endpoint elements $p_k$ and $q_1$, although
neither occurs in a nonabsorbing triple product of $T_k$. Thus
$T_k$ is a subsemiring of $T_{k+1}$ and
$T_\infty=\bigcup_{k\ge2}T_k$.

\begin{lemma}\label{lem:path-flat}
The algebras $T_k$ and $T_\infty$ are flat ai-semirings. They satisfy
$x^2\approx x^3$, and every product of four elements equals $\infty$.
Moreover
\[
 \mathcal U=\Var(\{T_k:k\ge2\}).
\]
\end{lemma}
\begin{proof}
First consider distributivity.
In each row and column of the multiplication table, all values
other than $\infty$ are distinct. This is exactly the
$\infty$-cancellation property: if $ab=ac\ne\infty$ then
$b=c$, and symmetrically on the right. For flat addition,
$\infty$-cancellation ensures that multiplying the sum of two
distinct elements agrees with the sum of their products. The
case of equal summands is immediate from idempotence.
Hence both distributive laws hold.

\medskip\noindent We next verify associativity and nilpotence.
The only nonabsorbing products of length three are
\[
 (u_i e)u_{i+1}=u_i(eu_{i+1})=\omega.
\]
Conversely, a triple product can be nonabsorbing only if its
first two factors form $u_ie=p_i$ and its last two form
$eu_{i+1}=q_{i+1}$, so the two possible associations agree.
Every other triple product evaluates to $\infty$ under both
associations, proving associativity. All squares equal $\infty$,
and every product of four factors is $\infty$.

\medskip\noindent Finally, consider the generated variety.
Each finite set of elements of $T_\infty$ is contained in some
$T_k$. Since an identity involves only finitely many variables,
it holds in $T_\infty$ exactly when it holds in every $T_k$;
the asserted equality of generated varieties follows.
\end{proof}

For later use, if $G$ is a finite directed path with at least two
vertices (all edges coherently oriented), or a directed cycle of
length at least four, write $T(G)$ for the semiring with elements
$\infty,\omega,e,u_i,p_i,q_i$ indexed by the vertices of $G$,
with $u_i e=p_i$ and $eu_i=q_i$. The products
$p_i u_j$ and $u_iq_j$ both equal $\omega$ if and only if
$i\to j$ is an edge of $G$; all other products are $\infty$. Thus $T_k=T(P_k)$, where $P_k$
has $k$ vertices, and we put $C_m=T(\vec C_m)$ for the directed
$m$-cycle, $m\ge4$.

\begin{proposition}\label{lem:path-embedding}
For every finite directed path $G$ with at least two vertices
and coherently oriented edges, and for every directed cycle $G$
of length at least four,
\[
 T(G)\in\W_2.
\]
Consequently $\mathcal U\subseteq\W_n$ for all $n\ge2$.
\end{proposition}
\begin{proof}
Let $G$ have vertex set $\{1,\ldots,k\}$ and put
$S=\widehat A_2=S(abba)^0$. We compare the kernels of
evaluation in $S$ and in $T(G)$ on the same generating
variables. Denote the newly adjoined zero by $0$, so that $0\ne\infty$.
We construct a subsemiring of a finite power of $S$ that
admits $T(G)$ as a homomorphic image. The required power is
$S^{1+k+\binom{k}{2}}$. Define vectors $g_1,\ldots,g_k,h$ as follows.
In a common coordinate, put $g_i=a$ for every $i$ and $h=bb$.
In the coordinate labelled $\{i\}$, put $g_i=a$, all other
$g_j=0$, and $h=bb$. In the coordinate labelled $\{i,j\}$,
when $i\to j$, put $g_i=a$, $g_j=b$, $h=b$, and all other $g_s=0$;
reverse the assignments to $g_i,g_j$ when $j\to i$.
If $i,j$ are nonadjacent, put $g_i=g_j=a$, $h=bb$, and the remaining
$g_s=0$. These instructions are unambiguous because $G$ has
neither loops nor opposite directed edges; for a directed path the
coherent orientation hypothesis is used throughout.

Let $H=\langle g_1,\ldots,g_k,h\rangle\le S^{1+k+\binom{k}{2}}$.
Our claim is that the assignment
\[
 g_i\longmapsto u_i\quad(1\le i\le k),
 \qquad h\longmapsto e
\]
induces a surjective homomorphism $H\twoheadrightarrow T(G)$.
To prove that the assignment is well defined, it suffices to
show that whenever two polynomial terms have the same value in
$H$, they also have the same value under the corresponding
assignment in $T(G)$.
Every term on $g_1,\ldots,g_k,h$ expands into a nonempty sum of
multiplicative words. In the common coordinate, its potentially
nonabsorbing words have exactly the five forms
\[
 g_i,\qquad h,\qquad g_i h,\qquad h g_i,\qquad g_i h g_j,
\]
All other words have common-coordinate value $\infty$. For the five listed forms, the coordinate values and their images in $T(G)$ are
\[
\begin{array}{c|ccccc}
 \text{word}&g_i&h&g_ih&hg_i&g_ihg_j\\ \hline
 \text{common coordinate}&a&bb&abb&bba&abba\\
 \text{image}&u_i&e&p_i&q_i&
 \begin{cases}\omega,&i\to j,\\\infty,&\text{otherwise.}\end{cases}
\end{array}
\]
Consequently a term with nonabsorbing image has one
of these five common-coordinate types. A term with two different
types, or with any common-coordinate absorbing word, has absorbing
image and absorbing common coordinate, so cannot coincide as a vector
with a term having nonabsorbing image.

\medskip\noindent We next distinguish the values represented within each
common-coordinate type. For the first four
types, the $k$ singleton coordinates record the relevant
indices: at coordinate $\{i\}$ the relevant
value is respectively $a,bb,abb,bba$ for the indicated index
(the common $h$ has value $bb$ in every singleton coordinate),
whereas indices absent from that coordinate contribute the added
zero $0$. Thus a term of type $g_i$, $g_i h$, or $h g_i$ has a
nonabsorbing image only if all its words carry one and the same
index; two different nonabsorbing images cannot yield the same vector.
The type $h$ has only the image $e$.

\medskip\noindent For the fifth type, we must distinguish
edges from nonedges: a term has nonabsorbing image precisely
when every word $g_i h g_j$ corresponds to an edge $i\to j$;
its image is then $\omega$. Suppose the term includes a non-edge
$g_i h g_j$. If $i=j$, at the singleton coordinate $\{i\}$
this word equals $abba$, while every edge word equals $0$.
If $i\ne j$ are nonadjacent, their pair coordinate has the same
property. If the reverse edge $j\to i$ is present, at its pair
coordinate the non-edge word $g_i h g_j$ has value $bba$,
while the unique edge word supported on the pair has value $abb$;
all other edge words give $0$. In all three cases, adding the
non-edge word changes the coordinate value of any sum consisting
only of edge words: the result is respectively $abba$ or $\infty$
(and in the reversed-edge case $bba$ or $\infty$), never the
edge-only value $0$ or $abb$. Hence a term with absorbing image
cannot share its vector value with a term whose image is $\omega$.

\medskip\noindent We now complete the HSP construction.
The coordinate tests prove the required inclusion of evaluation
kernels, and hence the map on generators extends uniquely to a
well-defined semiring homomorphism $H\to T(G)$. Its image
contains $u_i$ and $e$, and hence
$p_i,q_i,\omega$ and $\infty$, so the map is surjective.
Thus $T(G)\in\mathrm{HS}(S^{1+k+\binom{k}{2}})$.

In particular $T_k\in\W_2$ for all $k\ge2$ and hence
$\mathcal U\subseteq\W_2$ at the level of generated varieties. Finally,
$\widehat A_2$ is a subsemiring of $\widehat A_n$ obtained from
the central factor $abba$ of $\mword_n$ (renaming its two
letters). This gives $\W_2\subseteq\W_n$.
\end{proof}

\subsection{Nonfinite basis and separation}

\begin{lemma}\label{lem:path-NFB}
The variety $\Var(T_\infty)$ is nonfinitely based.
\end{lemma}
\begin{proof}
We begin with a family of identities valid in $T_\infty$.
For $m\ge1$ consider
\begin{equation}\label{eq:path-cycle}
 \chi_m:\quad \sum_{i=1}^{m}x_i yx_{i+1}\approx z^3,
 \qquad x_{m+1}=x_1.
\end{equation}
It holds in $T_\infty$: if the left-hand side were different from
$\infty$, all its summands would equal $\omega$. The multiplication
table then forces $y=e$, each $x_i=u_{j_i}$ and
$j_{i+1}=j_i+1$, including the return from $j_m$ to $j_1$.
This is impossible. On the other hand, $C_m$ violates $\chi_m$
for $m\ge4$, under $x_i=u_i$ and $y=e$.

\medskip\noindent We then examine the small generated subsemirings of the countermodels.
Fix $r<m$. Apart from $\infty,\omega,e$, each of any
$r$ chosen elements of $C_m$ involves at most one vertex index.
If $S$ is the set of
indices occurring in $r$ chosen generators, their generated
subsemiring is contained in
\[
 \{\infty,\omega,e\}\cup\{u_i,p_i,q_i:i\in S\}.
\]
Choose a vertex outside $S$ and cut the directed cycle there.
The remaining vertices form a directed path, and translation
along this path, preserving gaps between the indices in $S$,
embeds the generated subsemiring into $T_\infty$. Thus every $r$-generated subsemiring of $C_m$ belongs to
$\mathcal U$, and consequently $C_m$ satisfies every identity
of $T_\infty$ involving at most $r$ variables. If
$\mathcal U$ had a finite basis, choose $r$ bounding
the number of its variables and then $m>\max\{r,3\}$.
The semiring $C_m$ would satisfy the basis but violate
\eqref{eq:path-cycle}, a contradiction.
\end{proof}

\begin{proposition}\label{prop:path-separation}
For every $n\ge2$,
$\mathcal U\nsubseteq\V_n$.
\end{proposition}
\begin{proof}
Let
\begin{equation}\label{eq:path-separator}
 x_1yx_2+x_2yx_3+x_3yx_4
 \approx x_1yx_2+x_2yx_3+x_3yx_4+x_1yx_4.
\end{equation}
We verify the identity in $A_n$ directly from the factor
multiplication. If one of the first three summands is $\infty$,
their sum is $\infty$, and absorption makes the equation true.
Otherwise the three values must coincide with one nonabsorbing
factor $\mathbf{F}$ of $\mword_n$, by flatness.
Express their values as $\mathbf{U}_1\mathbf{V}\mathbf{U}_2=\mathbf{U}_2\mathbf{V}\mathbf{U}_3=\mathbf{U}_3\mathbf{V}\mathbf{U}_4=\mathbf{F}$, with all
$\mathbf{U}_i,\mathbf{V}$ nonempty words. Comparing the lengths of the first and second factorizations,
and then of the second and third, yields
$\ell(\mathbf{U}_1)=\ell(\mathbf{U}_3)$ and $\ell(\mathbf{U}_2)=\ell(\mathbf{U}_4)$.
The first and third expressions represent the very same word
$\mathbf{F}$. Their prefixes of length $\ell(\mathbf{U}_1)=\ell(\mathbf{U}_3)$
therefore agree, giving $\mathbf{U}_1=\mathbf{U}_3$; comparing suffixes in the
same way gives $\mathbf{U}_2=\mathbf{U}_4$. Consequently
$\mathbf{U}_1\mathbf{V}\mathbf{U}_4=\mathbf{U}_1\mathbf{V}\mathbf{U}_2=\mathbf{F}$, which is exactly the required
absorbed fourth summand.
In $T_4$, however, put $x_i=u_i$ and $y=e$. The first three
summands all equal $\omega$, whereas $u_1eu_4=\infty$.
Thus \eqref{eq:path-separator} fails in $T_4$, and the assertion follows from
$T_4\in\mathcal U$.
\end{proof}

\subsection{Bounded paths and hereditary finite basability}

To establish the limit property, we place every proper
subvariety of $\mathcal U$ inside a finitely based variety
with a finite subvariety lattice.
We use two results of Jackson--Ren--Zhao
\cite[Lemmas~2.1--2.2]{JRZ22}. Every subdirectly
irreducible member of the variety generated by flat
ai-semirings is flat; we require only this direction,
not the converse assertion that every flat ai-semiring is
subdirectly irreducible. The variety is axiomatized,
within the class of ai-semirings, by the $16$ identities
\begin{equation}\label{eq:path-flat-identity}
 x_1ux_2+y_1uy_2+y_1vy_2
 \approx x_1vx_2+y_1uy_2+y_1vy_2,
\end{equation}
where each of the four contexts $x_1,x_2,y_1,y_2$ may
independently be omitted. Flat ai-semirings are also
\emph{$\infty$-cancellative}: whenever $ab=ac\ne\infty$, one has $b=c$; likewise
$ba=ca\ne\infty$ entails $b=c$.
Here omission of a context in \eqref{eq:path-flat-identity} is an
abbreviation for an identity in our unchanged language, not an
introduction of a multiplicative identity.

Define the ambient variety $\mathcal F$ by the following finite
set of identities. Throughout this paragraph, $\kappa$ denotes
a \emph{term abbreviation} for the common value of all cubes,
not a new constant symbol:
\begin{equation}\label{eq:path-base-F}
 x^3\approx y^3,\quad x+x^3\approx x^3,\quad
 x^2\approx x^3,\quad x^4\approx x^3,\quad
 xyzt\approx x^3.
\end{equation}
In $\mathcal F$, $\kappa$ is both the additive maximum and
multiplicative zero, represented by a term rather than a named
constant. In the next family, $\mathbf{w}\in\{ayr,rya\}$ and
$\mathbf{w}'\in\{bys,syb\}$ abbreviate the corresponding finitely many
word identities; $\mathbf{w},\mathbf{w}'$ are not new variables. In addition to the ai-semiring
axioms and \eqref{eq:path-flat-identity} and
\eqref{eq:path-base-F}, require
\begin{align}
 xyz+uvw&\approx xyz+uvw+xvz, \label{eq:path-M}\\
 xyz+yt&\approx\kappa,\qquad xyz+ty\approx\kappa,
   \label{eq:path-E}\\
 \mathbf{w}+\mathbf{w}'+ab&\approx\kappa, \label{eq:path-V}\\
 \mathbf{w}+ab+bd&\approx\kappa,\qquad
 \mathbf{w}+db+ba\approx\kappa, \label{eq:path-P}\\
 xy+yz+zt&\approx\kappa. \label{eq:path-Z}
\end{align}
Here $\kappa$ abbreviates the cube of a fresh variable;
all the displayed identities belong to the constant-free language.

We verify the five additional families in $T_\infty$.
A sum in a flat semiring is different from $\infty$ only
when all its summands have the same non-top value. The only
non-top triple products in $T_\infty$ are
$u_i e u_{i+1}=\omega$, and the only binary products
with value $\omega$ are
\[
 p_i u_{i+1}=\omega=u_iq_{i+1}\qquad(i\ge1).
\]
For~\eqref{eq:path-M}, the identity is automatic if
one of $xyz,uvw$ equals $\infty$. Otherwise both are
$\omega$, and their middle factors $y,v$ are both $e$;
thus $xvz=xyz=\omega$.
For~\eqref{eq:path-E}, if $xyz=\omega$, then $y=e$;
no product $et$ or $te$ equals $\omega$, so each displayed
sum equals $\infty$. For~\eqref{eq:path-V}, if both
$\mathbf{w},\mathbf{w}'$ have value $\omega$, their distinguished variables
$a,b$ are vertices $u_i,u_j$, respectively, and
$ab=\infty$. This argument applies separately to both
placements of $a$ and both placements of $b$ in the
abbreviations defining $\mathbf{w},\mathbf{w}'$.
For~\eqref{eq:path-P}, a non-top $\mathbf{w}$ again forces
$a=u_i$ for some $i$. The only possible non-top $ab$ is
$u_iq_{i+1}=\omega$, whereas every product with left
factor $q_{i+1}$ equals $\infty$; hence $ab$ and $bd$
cannot both equal $\omega$. Similarly, if $ba=\omega$,
then $b=p_{i-1}$ (when defined), and no product $db$
equals $\omega$. Thus each of the two identities in
\eqref{eq:path-P} holds. Finally, the directed graph of
the binary products with value $\omega$ has only the
one- and two-edge chains
$p_i,u_{i+1},q_{i+2}$, whose consecutive products are
$\omega$ (with any missing terminal edge omitted).
In particular, there is no chain of three such products,
which proves~\eqref{eq:path-Z}. Hence
\begin{equation*}
 \mathcal U\subseteq\mathcal F.
\end{equation*}

The identities $\chi_m$ in~\eqref{eq:path-cycle} involve a closed
cycle, whereas the following identities describe open directed
paths. They are not interchangeable by reindexing.
For $N\ge1$, define $\mathcal D_N\subseteq\mathcal F$ by
\begin{equation*}
 \Phi_N:\quad \sum_{i=0}^{N}x_i y x_{i+1}\approx\kappa.
\end{equation*}
The displayed axioms form a finite basis for $\mathcal D_N$.
The same variety is locally finite: by \eqref{eq:path-base-F}, every term is a sum of words
of length at most three or has value $\kappa$. Thus an
$r$-generated member has at most $2^{r+r^2+r^3}+1$ elements.

\begin{proposition}\label{lem:path-components}
Let $S$ be a nontrivial finite subdirectly irreducible member
of $\mathcal F$. There is a unique element $\omega\ne\infty$
with $\omega S=S\omega=\{\infty\}$.
If $S$ has a nonabsorbing triple product, all such products equal
$\omega$ and have the same middle factor $e$. Put
\[
 A=\{a:a\text{ is an endpoint of a nonabsorbing triple}\}.
\]
Define a directed edge $a\to b$ on $A$ precisely when
$aeb=\omega$.
\[
 \mathcal B=\{p_a=ae:a\to b\text{ for some }b\},\qquad
 \mathcal C=\{q_b=eb:a\to b\text{ for some }a\},
\]
and let $\mathcal R=S\setminus(\{\infty,\omega,e\}\cup A\cup \mathcal B\cup \mathcal C)$.
These sets form a disjoint union
\[
 S=\{\infty,\omega,e\}\,\dot\cup\, A\,\dot\cup\,\mathcal B
   \,\dot\cup\,\mathcal C\,\dot\cup\,\mathcal R.
\]
The digraph $(A,\to)$ has indegree and outdegree at most one.
Each of its finite acyclic components is a directed path; it
may have at most one incoming $\mathcal R$-leaf at its initial vertex
and one outgoing $\mathcal R$-leaf at its terminal vertex.
The other members of $\mathcal R$ form disjoint directed paths of one or
two edges under $r s=\omega$. Products of elements from
distinct components are $\infty$, apart from interactions
with the common element $e$.
If there is no nonabsorbing triple product, the digraph
with vertex set $S\setminus\{\infty,\omega\}$ and edge relation
$x\rightsquigarrow y$ defined by $xy=\omega$ is a disjoint union of directed
paths with one or two edges.
\end{proposition}
\begin{proof}
\emph{1. The unique terminal element.}
Fix $s\ne\infty$. Starting with the one-factor word $s$,
extend it on either side as long as its value remains
non-top. No non-top product has four factors, so a word of
maximal possible length is reached after at most three
factors. Its value $w$ satisfies $wS=Sw=\{\infty\}$;
otherwise another non-top extension would be possible.
Thus for every non-top element $s$ there is a finite sequence
of non-top values
\[
 s=s_0,s_1,\ldots,s_r=w,\qquad wS=Sw=\{\infty\},
\]
where each $s_{i+1}$ is obtained by multiplying $s_i$
on the left or right by an element of $S$.
If $w\ne w'$ are terminal, then
$\{\infty,w\}$ and $\{\infty,w'\}$ are ideals: flat addition
keeps sums with these sets inside them, and multiplication
always gives $\infty$. Their Rees congruences identify respectively just the
pairs $(\infty,w)$ and $(\infty,w')$ (and their reverses),
so each is nontrivial while their intersection is the diagonal
congruence. This contradicts subdirect irreducibility.
Thus the terminal element $\omega$ is unique. Every
nonabsorbing triple product is terminal and hence equals $\omega$.

\medskip\noindent\emph{2. Common middle factor of nonabsorbing triples.}
Suppose now that a nonabsorbing triple exists. If
$abc=a'b'c'=\omega$, substitute the two triples into
\eqref{eq:path-M}. Since the first two summands both
have value $\omega$, flat addition forces the added
summand $ab'c$ to have that same value $\omega$:
otherwise the sum would be $\infty$. From $abc=ab'c=\omega\ne\infty$, left
$\infty$-cancellation gives $bc=b'c$; right
$\infty$-cancellation then gives $b=b'$. Call this common
middle factor $e$. Its square is $\infty$, and neither $e$
nor any endpoint of a nonabsorbing triple can be the value
of a nonabsorbing binary product: substituting such a
binary product into the indicated triple would give a
nonabsorbing product of four factors.

\medskip\noindent Let $xy=d\notin\{\infty,\omega\}$. Since $d$ is not
terminal, some $z$ satisfies $dz=\omega$ or $zd=\omega$.
If $xyz=\omega$, uniqueness of the middle factor gives
$y=e$, $x,z\in A$, and $d=xe$. If $zxy=\omega$, it gives
$x=e$, $z,y\in A$, and $d=ey$.
Thus every nonterminal binary value has the form $ae$ or
$eb$ with $a,b\in A$. These two forms cannot represent
the same nonterminal element. Indeed, suppose
$d=ae=eb'$ for $a,b'\in A$. From $d=ae$ and the
nonterminality of $d$, its extension to a nonabsorbing
triple must be on the right: there is $b\in A$ with
\[
              db=(ae)b=\omega.
\]
But then $(eb')b=\omega$ is a nonabsorbing triple with
middle factor $b'\ne e$, contrary to the uniqueness of
the middle factor $e$. Hence the two possibilities are
disjoint.
Moreover, applying \eqref{eq:path-E} to a triple with
middle factor $e$ shows that $et\ne\omega$ and
$te\ne\omega$ for every $t$. Thus $ae$ or $eb$, if
nonabsorbing, is nonterminal. It must extend in a triple:
$ae$ cannot extend to the left, as the resulting triple
would have middle factor $a\ne e$, and $eb$ cannot extend
to the right. Consequently, $ae$ is nonabsorbing if and only if
$a\to b$ for some $b$. Likewise, $eb$ is nonabsorbing
if and only if $a\to b$ for some $a$.
The $p_a$ are pairwise distinct, as are the $q_b$, by
cancellation. Nor can $p_a=q_b=d$: a nonterminal element
$d$ extends to $dz=\omega$ or $zd=\omega$; the first
would make $ebz=\omega$ with middle factor $b\ne e$,
while the second would make $zae=\omega$ with middle
factor $a\ne e$. The same length-four argument above
separates $\mathcal B,\mathcal C$ from $A\cup\{e\}$, and $\omega$ is
terminal. This proves the asserted disjointness.

Cancellation makes the edge relation functional in both
directions: if $a\to b$ and $a\to c$, then $b=c$;
if $a\to c$ and $b\to c$, then $a=b$.
Associativity also identifies the two products associated
with an edge. Whenever $p_a$ and $q_b$ are defined, each
of the following conditions holds if and only if the others do:
\begin{equation}\label{eq:path-edge-equivalence}
 p_a b=\omega,\qquad a\to b,\qquad a q_b=\omega.
\end{equation}

\medskip\noindent\emph{3. Classification of nonabsorbing products.}
We classify products by ordered factor type. Apart from
$Ae\subseteq \mathcal B\cup\{\infty\}$,
$eA\subseteq \mathcal C\cup\{\infty\}$ and
\eqref{eq:path-edge-equivalence}, only $A\mathcal R$, $\mathcal R A$ and $\mathcal R\mathcal R$
may take the nonabsorbing value $\omega$. To exclude the
other cases, note that $AA=\infty$ by~\eqref{eq:path-V}
and the nonterminal-value classification. Products with two
factors in $\mathcal B\cup \mathcal C$ have length at least four.
The types $A\mathcal B,\mathcal C A,e\mathcal B,\mathcal C e,\mathcal R\mathcal B,\mathcal C\mathcal R$ would give a triple with
middle factor other than $e$; $\mathcal B e,e\mathcal C$ contain $e^2$;
and $\mathcal B\mathcal R,\mathcal R\mathcal C$ would force an $\mathcal R$-element to be a triple
endpoint. Finally, $e\mathcal R,\mathcal R e$ cannot be nonterminal and are
excluded by~\eqref{eq:path-E}. Products involving
$\omega$ or $\infty$ equal $\infty$. Thus the complete
table is as follows, where $\omega/\infty$ denotes the
two permitted possibilities subject to the edge relations:
\[
\begin{array}{c|ccccc}
\cdot&A&e&\mathcal B&\mathcal C&\mathcal R\\\hline
 A&\infty&\mathcal B/\infty&\infty&\omega/\infty&\omega/\infty\\
 e&\mathcal C/\infty&\infty&\infty&\infty&\infty\\
 \mathcal B&\omega/\infty&\infty&\infty&\infty&\infty\\
 \mathcal C&\infty&\infty&\infty&\infty&\infty\\
 \mathcal R&\omega/\infty&\infty&\infty&\infty&\omega/\infty
\end{array}
\]
To read the table, note for example that a non-top
product in $\mathcal B A$ must be $p_ab=(ae)b=\omega$, hence
comes from the triple $aeb$; conversely every edge
$a\to b$ forces this product and also $a(eb)=\omega$.
For a factor in $\mathcal C$ on the left, $(eb)t$ would be a
non-top triple with middle factor $b\ne e$; for a factor
in $\mathcal B$ on the right, $t(ae)$ would have middle factor
$a\ne e$. These eliminate entire ordered rows or
columns, rather than relying on an implicit symmetry.
The table describes possible values, not independent choices
for its $\omega/\infty$ entries. The products involving $e$
and the edge products $p_ab,aq_b$ are determined by $\to$;
the only residual types are $A\mathcal R$, $\mathcal R A$, and $\mathcal R\mathcal R$, whose
nonabsorbing values must equal $\omega$. We now determine
which such products can occur.

\medskip\noindent\emph{4. Directed paths and possible attached leaves.}
It remains to restrict the terminal products of types
$A\mathcal R$, $\mathcal R A$, and $\mathcal R\mathcal R$. Write $x\rightsquigarrow y$
when $xy=\omega$. Fix a nonabsorbing triple $aeb=\omega$ with $a,b\in A$.
For any $a,b\in A$, choose a non-top triple with $a$ as
an endpoint and another with $b$ as an endpoint. Each can
be written in one of the forms $ayr,rya$ and $bys,syb$,
respectively. If $ab=\omega$, all three summands in the
corresponding instance of~\eqref{eq:path-V} equal
$\omega$, contradicting that identity. Hence $ab\ne\omega$.
This argument covers all four choices of endpoint positions,
not only triples having $a$ and $b$ in their first positions. No such product can be a
nonterminal value either, by the already established
classification of nonterminal binary products. Thus $AA=\infty$.

The remaining possibilities are excluded
without any graph terminology. If $x\in \mathcal B$ and $xy\ne\infty$,
write $x=ae$; then $(ae)y=\omega$ and $aey=\omega$, so
$y\in A$. If $y\in \mathcal C$ and $xy\ne\infty$, write $y=eb$;
then $x(eb)=\omega$ and $xeb=\omega$, so $x\in A$.
If $x\in \mathcal C$, a nonabsorbing $xy$ would make $eby$ a triple
with middle factor $b\in A$, contrary to the uniqueness of
$e$. Dually, a nonabsorbing $xy$ with $y\in \mathcal B$ would make
$xae$ a triple with middle factor $a\in A$.
A nonabsorbing product involving $e$ can only be $ae\in \mathcal B$
or $eb\in \mathcal C$: the other alternatives either force a triple
with middle factor distinct from $e$, or contradict
\eqref{eq:path-E} when their value is $\omega$.
Finally, no element of $\mathcal R$ can be a factor
of a nonterminal product. To see this, suppose
$xy=d\notin\{\infty,\omega\}$. By the preceding
extension argument, either $xyz=\omega$ for some $z$,
in which case $y=e$ and $x\in A$, or $zxy=\omega$,
in which case $x=e$ and $y\in A$. Thus neither $x$
nor $y$ belongs to $\mathcal R$. In particular, for
$r\in\mathcal R$, neither $er$ nor $re$ can be
nonterminal; neither can equal $\omega$, by
\eqref{eq:path-E}. Hence
\[
                 e\mathcal R\subseteq\{\infty\},\qquad\mathcal R e\subseteq\{\infty\}.
\]
The same argument excludes $\mathcal R$ from all other
nonterminal product types. This proves that the displayed
multiplication table lists all nonabsorbing binary products.

Cancellation makes the relation $\rightsquigarrow$ have at most
one outgoing and one incoming edge at each vertex. Identity
\eqref{eq:path-Z} forbids a walk of three such edges: if
$x\rightsquigarrow y\rightsquigarrow z\rightsquigarrow t$,
then $xy+yz+zt=\omega\ne\kappa$. Since every square is
$\kappa=\infty$, loops are impossible. A two-cycle would
contradict \eqref{eq:path-Z} by taking $(x,y,z,t)=(r,s,r,s)$;
a cycle of length at least three contains a three-edge walk
and is also excluded.

Identity \eqref{eq:path-P} forbids a two-edge
$\rightsquigarrow$-walk whose initial or terminal vertex
belongs to $A$. To see the first assertion, let $a\in A$
and choose one of the words $\mathbf{w}=ayr$ or $\mathbf{w}=rya$ which
expresses a nonabsorbing triple with $a$ in the indicated
end position. If $a\rightsquigarrow b\rightsquigarrow d$,
then $\mathbf{w}+ab+bd=\omega\ne\kappa$, contradicting the first
identity in \eqref{eq:path-P}. If instead
$d\rightsquigarrow b\rightsquigarrow a$, use
$\mathbf{w}+db+ba$ and the second identity. The choice of $\mathbf{w}$ is
possible because every $a\in A$ is an endpoint of a triple.
If $r\rightsquigarrow a$ for $r\in \mathcal R$, $a\in A$, then
$a$ has no incoming edge $b\to a$, since otherwise
$p_ba=ra=\omega$ and right cancellation gives
$r=p_b\in \mathcal B$. As every member of $A$ occurs in some
triple, $a$ is therefore the initial vertex of its
$\to$-component. Likewise, $a\rightsquigarrow r$
can occur only at the terminal vertex. There is at most
one such leaf at either end, and an attached $\mathcal R$-vertex
has no further incident $\rightsquigarrow$-edge, by the
two-edge restriction (and cancellation).
Every remaining $r\in \mathcal R$ has an incident
$\rightsquigarrow$-edge: otherwise $r$ would be another
terminal element. Its edges join it only to members of
$\mathcal R$, by the preceding multiplication classification.
Such components are directed paths of one or two edges.
Here a component includes its $A$-vertices, the associated
$p_a,q_b$, and any $\mathcal R$-leaves attached to its endpoints;
other $\mathcal R$-paths are separate components. Every non-top
product listed in the table either has both non-$e$ factors
in the same such component or involves $e$. Thus products
between distinct components are $\infty$, with $e$
serving as the only common nonterminal multiplier.

Thus no other residual nonabsorbing product is possible:
$\mathcal R A$ and $A\mathcal R$ can attach leaves only at the initial and
terminal vertices, respectively, while $\mathcal R\mathcal R$-edges form
directed paths of at most two edges. The exclusions above
follow from cancellation and \eqref{eq:path-P}--\eqref{eq:path-Z}.

\medskip\noindent\emph{5. The case without nonabsorbing triples.}
If there is no nonabsorbing triple,
every nonabsorbing binary product must equal $\omega$,
since any other binary value could be extended to such a triple.
Cancellation and \eqref{eq:path-Z} make the resulting
digraph a union of directed paths of at most two edges;
all vertices outside $\{\infty,\omega\}$ lie on such an
edge. Indeed, an element outside this pair with no
incident non-top product would itself be terminal,
contradicting uniqueness of $\omega$. This argument
is independent of the case with a common middle factor.
\end{proof}

A finite semiring $S$ is called \emph{critical} if it does
not belong to the variety generated by its proper subsemirings
and proper homomorphic images.

\begin{proposition}\label{lem:path-DN-finite-lattice}
For every $N\ge1$, $\mathcal D_N$ has finitely many
subvarieties, and each of them is finitely based.
\end{proposition}
\begin{proof}
\emph{1. Bounded component types.}
Let $S$ be a finite nontrivial subdirectly irreducible
member of $\mathcal D_N$. The graph $(A,\to)$ of
Proposition~\ref{lem:path-components}, when present, has no
directed walk $a_0\to a_1\to\cdots\to a_{N+1}$.
Indeed, the substitution $x_i\mapsto a_i$ and $y\mapsto e$
would send every summand $x_i y x_{i+1}$ of $\Phi_N$
to $\omega$, and hence its left side to $\omega\ne\infty=\kappa$.
This contradicts $S\models\Phi_N$.
Thus its components are directed paths of at most $N$
edges. A component with $j$ edges has $j+1$ vertices,
$j$ elements of $\mathcal B$, $j$ of $\mathcal C$ and at most two attached
$\mathcal R$-leaves, hence at most $3j+3$ elements outside the
common part $K=\{\infty,\omega,e\}$. Up to
isomorphism fixing $K$, there are at most four types for
each $1\le j\le N$, according to the two optional
leaves. There are also the one- and two-edge components
made entirely of $\mathcal R$-vertices. In the no-triple case the
common part is $K=\{\infty,\omega\}$ and only the last two
types occur. Therefore there are at most $4N+2$ types
in either case, each with at most $3N+3$ elements outside $K$.
The set $K$ is a common subsemiring, not the kernel of a
homomorphism. Both operations on $K$ close in $K$;
the component classification shows that products of elements
from distinct components are $\infty$, whereas operations
within one component, including those with $e\in K$, stay
in that component together with $K$. For addition, distinct
non-top elements likewise sum to $\infty$.

\medskip\noindent\emph{2. Removing repeated components.}
Suppose a component type occurs $r\ge3$ times.
We show that one of these components can be removed without
changing the generated variety. Components here
are the blocks outside $K$, with the same common part $K$;
an occurrence of a type means an isomorphism fixing $K$
and preserving all products with elements of $K$.
Let $S_{\mathrm{red}}$ be obtained by deleting the elements outside $K$ of
one such component. The multiplication description in
Proposition~\ref{lem:path-components} shows that $S_{\mathrm{red}}$ is a
proper subsemiring. Choose two retained components of
this type, numbered $1,2$, and distinct words
$c_1,\ldots,c_r\in\{1,2\}^d$ with $2^d\ge r$.
Fix isomorphisms of all $r$ components onto each of the
two retained components, acting identically on $K$.
For an element $s$ in component $i$ let the $t$-th
coordinate of $\iota(s)$ be its image in component
$c_i(t)$. Map the elements of $K$ and of other types
diagonally. The resulting map of underlying sets
$\iota:S\to(S_{\mathrm{red}})^d$ is injective, but need not
preserve the operations.
Indeed, distinct original components may share the same retained
component in some coordinates, so a product that is $\infty$
in $S$ need not have value $\iota(\infty)$ in every coordinate.
The homomorphism will be obtained only after taking a Rees quotient.
Let
$C=\langle\iota(S)\rangle$ and set
\[
 I=\{v\in C:\text{some coordinate of }v\text{ equals }\infty\}.
\]
The vector $(\infty,\ldots,\infty)=\iota(\infty)$ lies
in $I$, whereas $\iota(s)\notin I$ whenever
$s\ne\infty$, so $I$ is nonempty and proper.
Since a coordinate equal to $\infty$ stays equal to
$\infty$ under either addition or multiplication with
\emph{any} other coordinate, $I$ is absorbing under both
operations against elements of $C$. The relation identifying
all members of $I$ and leaving $C\setminus I$ as
singletons is therefore a congruence for each basic operation;
the Rees quotient $C/I$ is a semiring.

\medskip\noindent\emph{3. Verification of the quotient homomorphism.}
We verify the induced map into the Rees quotient for both
basic operations and then prove injectivity.
For $\circ\in\{+,\cdot\}$, first suppose
$s\circ t\ne\infty$. Then either $s,t$ belong to the same
component together with $K$, or one belongs to $K$. The
distinguished common part is fixed pointwise by every chosen
component isomorphism; in particular, this includes the cases
$ea$ and $ae$, as well as products involving $\omega$.
On each coordinate,
the chosen component isomorphism preserves $\circ$;
hence
\[
 \iota(s)\circ\iota(t)=\iota(s\circ t).
\]
If $s\circ t=\infty$, either the corresponding
operation is already $\infty$ inside one component or
with $K$ (and hence is $\infty$ in every
coordinate), or $s,t$ belong to different components.
For two \emph{different original components} of the same
repeated type, their codes differ in some coordinate. In that
coordinate their images outside $K$ lie in different retained
components, so both their sum and product are $\infty$.
Thus distinct binary codes are essential: assigning every
original component to a single retained copy would not preserve
the absorbing products. Components of different types remain
distinct in every coordinate and give the same conclusion.
Thus in all cases $\iota(s)\circ\iota(t)\in I$.
The case $s=t$ for addition follows by idempotence.

Thus, for $\circ\in\{+,\cdot\}$, we have established
\[
\begin{cases}
 \iota(s)\circ\iota(t)=\iota(s\circ t),&s\circ t\ne\infty,\\
 \iota(s)\circ\iota(t)\in I,&s\circ t=\infty.
\end{cases}
\]
All products whose original value is $\infty$
therefore become the ideal class of the Rees quotient. No $\iota(s)$ with $s\ne\infty$ belongs to $I$:
component isomorphisms take non-top elements to non-top elements,
and $K$ is mapped diagonally. If $s,t$ are distinct members
of a single component, injectivity of its component isomorphism
separates their images in every coordinate. If they belong to
different repeated components, their distinct binary codes separate
them in at least one coordinate, where their images lie in
different retained components; distinct component types remain
separate in every coordinate. Elements of $K$ are
also distinct from images outside $K$. Thus different non-top
elements have different images in the Rees quotient.
Define
\[
 \pi:S\longrightarrow C/I,\qquad
 \pi(s)=
 \begin{cases}
  I,&s=\infty,\\
  [\iota(s)],&s\ne\infty.
 \end{cases}
\]
The two cases above yield, for both $\circ=+$ and
$\circ=\cdot$,
\[
 \pi(s\circ t)=\pi(s)\circ\pi(t):
\]
 if $s\circ t\ne\infty$,
its image remains a nonideal class and coordinatewise operations
agree; if $s\circ t=\infty$, at least one coordinate has
value $\infty$ and both sides are the ideal class.
The separation argument proves that $\pi$ is injective.
For surjectivity, we prove by induction on the formation
of terms over the generating set $\iota(S)$ that every
element $v\in C$ either lies in $I$ or equals
$\iota(s)$ for some $s\in S\setminus\{\infty\}$.
For a generator, this holds by definition of $\iota$
(the generator $\iota(\infty)$ lies in $I$).
Suppose it holds for $v_1,v_2\in C$ and put
$v=v_1\circ v_2$, with $\circ\in\{+,\cdot\}$.
If some $v_i\in I$, then $v\in I$ because $I$ is
absorbing under both operations. Otherwise write
$v_i=\iota(s_i)$ with $s_i\ne\infty$. The two-case
calculation above gives either
$v=\iota(s_1\circ s_2)$ when $s_1\circ s_2\ne\infty$,
or $v\in I$ when $s_1\circ s_2=\infty$.
This completes the induction and yields
\[
 C\setminus I=\iota(S\setminus\{\infty\}).
\]
Thus every nonideal class of $C/I$ is $\pi(s)$ for
some $s\ne\infty$, while the ideal class is
$\pi(\infty)$. Therefore $\pi$ is surjective, and
\[
 S\cong C/I\in\mathrm{HS}((S_{\mathrm{red}})^d).
\]
In particular, $S$ belongs to the variety generated by
the \emph{proper} subsemiring $S_{\mathrm{red}}$, not merely
to a class generated by another semiring of unbounded size.
Consequently a critical member of $\mathcal D_N$
cannot contain three components of the same type. The common part $K$, including $e$ when present, is retained.
The two surviving copies provide the distinct code symbols
needed to separate original components.

\medskip\noindent\emph{4. Uniform bound for critical members.}
We now obtain a uniform bound for critical algebras.
A critical finite semiring must be subdirectly
irreducible: otherwise two nontrivial congruences with diagonal
intersection provide a subdirect embedding into a product of
proper homomorphic images. This contradicts criticality.
The same argument applies to a finite subdirect
decomposition by selecting a finite separating family of kernels. Thus every critical member of
$\mathcal D_N$ satisfies
\[
\begin{aligned}
 |S|&\le |K|+2\cdot
       (\text{number of component types})
       \cdot(\text{maximum component size})\\
    &\le 3+2(4N+2)(3N+3).
\end{aligned}
\]
There are only finitely many isomorphism types below this bound
because the language has two binary operations and a finite
universe admits only finitely many such operation tables.
Every subvariety of the locally finite $\mathcal D_N$ is generated
by its finite members and, by induction on their orders, by its
critical members: each noncritical member belongs to the variety
generated by its proper subsemirings and homomorphic images.
Hence $\mathcal D_N$ has only finitely many subvarieties.

\medskip\noindent\emph{5. Finite subvariety lattices and finite bases.}
Let
$\mathcal X\subsetneq\mathcal Y$ be a cover relation
in the finite lattice of subvarieties of $\mathcal D_N$.
Choose an identity valid in $\mathcal X$ but failing in
$\mathcal Y$. Within $\mathcal Y$, the models of this
identity form a subvariety containing $\mathcal X$ but
properly contained in $\mathcal Y$; by the cover property,
this subvariety must equal $\mathcal X$.
Choose a finite saturated chain descending from
$\mathcal D_N$ to an arbitrary subvariety. One separating
identity for each cover in the chain, together with the finite
axiom set of $\mathcal D_N$, defines the final subvariety.
This proves its finite basability.
\end{proof}

\begin{theorem}\label{thm:path-limit}
The variety $\Var(T_\infty)$ is a limit variety.
Moreover, for each $n\ge2$,
\[
 \Var(T_\infty)\subseteq\W_n,
 \qquad \Var(T_\infty)\nsubseteq\V_n.
\]
\end{theorem}
\begin{proof}
We prove that every proper subvariety
$\mathcal X\subsetneq\mathcal U$ is finitely based.
Since $\mathcal X$ omits an identity of $T_\infty$, it omits
some $T_k$; otherwise it would contain every $T_k$ and hence
$\mathcal U$. Fix $k$ with $T_k\notin\mathcal X$.

\medskip\noindent We examine the finite subdirectly irreducible members of $\mathcal X$. 
Let $B$ be a finite subdirectly irreducible member of
$\mathcal X$. The ambient variety $\mathcal F$ and
Proposition~\ref{lem:path-components} apply. If $B$ has no
nonabsorbing triple product then $B\models\Phi_k$.
Otherwise its triple-edge graph has no directed cycle.
Indeed, a cycle of length $m$ would contradict
$\chi_m$, which belongs to $\Id(\mathcal U)$
by Lemma~\ref{lem:path-NFB}. Hence the finite components
of this functional graph are directed paths.

\medskip\noindent We show that a sufficiently long path would contain a copy of $T_k$.
Suppose, toward a contradiction, that a directed path has
$k+1$ edges,
$v_0\to v_1\to\cdots\to v_k\to v_{k+1}$.
The set
\[
 S=\{\infty,\omega,e\}\cup
   \{v_i,v_ie,ev_i:1\le i\le k\}
\]
is a subsemiring of $B$ isomorphic to $T_k$.
The edges incident with $v_0$ and $v_{k+1}$ ensure that
$ev_1$ and $v_ke$ are nonabsorbing. Cancellation and disjointness of $A,P,Q$
show that $|S|=3k+3$. By
\eqref{eq:path-edge-equivalence}, the only nonabsorbing
products in $S$ are $v_ie,ev_i$ and
$(v_ie)v_{i+1}=v_i(ev_{i+1})=\omega$ for $1\le i<k$.
All other products in $S$ are $\infty$ by the classification in
Proposition~\ref{lem:path-components}; flat addition is closed
because $\infty\in S$. The bijection from $T_k$ taking
$u_i$ to $v_i$, $p_i$ to $v_ie$ and $q_i$ to $ev_i$
is thus an isomorphism. This contradicts $T_k\notin\mathcal X$. The external vertices $v_0,v_{k+1}$ are not in $S$;
they serve only to ensure the presence of $q_1=ev_1$ and
$p_k=v_ke$. Thus $S$ contains all $3k+3$ distinct
elements of $T_k$.

\medskip\noindent The preceding contradiction excludes paths with $k+1$
edges in each finite subdirectly irreducible member of
$\mathcal X$. A nonabsorbing value of the left side of
$\Phi_k$ would require a directed walk with $k+1$ edges.
Therefore every finite subdirectly irreducible member of
$\mathcal X$ satisfies $\Phi_k$. By finite
subdirect representation and preservation of identities under
products and subalgebras, every finite member of
$\mathcal X$ satisfies $\Phi_k$. Since $\mathcal X\subseteq\mathcal F$ is locally
finite, every finitely generated member is finite; hence every
member satisfies $\Phi_k$. Therefore
$\mathcal X\subseteq\mathcal D_k$.
By Proposition~\ref{lem:path-DN-finite-lattice},
$\mathcal X$ is finitely based.

Lemma~\ref{lem:path-NFB} proves that $\mathcal U$
is not finitely based, while Proposition~\ref{lem:path-embedding}
and Proposition~\ref{prop:path-separation} give the asserted
inclusions and noninclusions. 
\end{proof}

\begin{corollary}\label{cor:path-distinct}
The limit variety $\mathcal U$ is distinct from
$\Var(SR_6)$ and is not contained in $\W_1$.
\end{corollary}
\begin{proof}
The following identity distinguishes the two varieties.
In $T_\infty$ every square equals $\infty$; moreover,
whenever $xy$ and $yx$ are both nonabsorbing, they cannot
take the same value. Consequently
\[
 xy+yx\approx yx+y^2
\]
holds in $\mathcal U$. It fails in $SR_6$:
with the usual table labels, $2\cdot6=6\cdot2=3$ whereas
$6^2=1=\infty$. The two limit varieties are therefore distinct.
Since $\W_1$ has only the limit subvariety $\Var(SR_6)$ by
Theorem~\ref{thm:unique-limit}, $\mathcal U$
cannot be a subvariety of $\W_1$.
\end{proof}

\begin{corollary}\label{cor:non-cross}
None of the varieties $\W_n$ ($n\ge1$) or $\W_\infty$
is a Cross variety. Furthermore, $\W_n$ is not a limit
variety for any $n\ge2$.
\end{corollary}
\begin{proof}
Recall that a Cross variety is finitely based, is generated
by a finite algebra, and has a finite subvariety lattice.
For
$n\ge2$, $\W_n$ is nonfinitely based by
Theorem~\ref{thm:main}, as is $\W_\infty$ by
Theorem~\ref{ind:thm:main}; hence none is Cross.
The variety $\W_1$ contains the nonfinitely based limit
subvariety $\Var(SR_6)$ by Theorem~\ref{thm:unique-limit}.
To justify the use of the finite-subvariety-lattice condition,
note that every subvariety of a finitely based variety with a
finite subvariety lattice is itself finitely based. Indeed, fix
a subvariety $\mathcal Y$ of such an ambient variety.
For each other subvariety not contained in $\mathcal Y$,
choose an identity holding in $\mathcal Y$ but failing in
that other subvariety. There are only finitely many such
subvarieties; adjoining all selected identities to a finite
basis for the ambient variety defines $\mathcal Y$. Thus $\W_1$ is not
Cross either.
Finally $\mathcal U\subsetneq\W_2$:
containment follows from Theorem~\ref{thm:path-limit}, and it
is proper because $\mathcal U$ satisfies
$x^2\approx y^2$ whereas $\widehat A_2$ does not.
Thus each $\W_n$ for $n\ge2$ contains a proper nonfinitely
based subvariety $\mathcal U$, so it cannot be a
limit variety.
\end{proof}

\Needspace{5\baselineskip}
\subsection{Comparison with the limit varieties of Shao--Gao}\label{subsec:limit-comparison}
Put $\mathcal A=\mathsf V(S(abc))$, where $a,b,c$ are distinct letters.
Put $L_{\mathrm{SG}}=\mathsf V(B_{15})$, and let
$\mathcal M$ be the limit variety whose existence is proved in
Shao--Gao~\cite[Theorem~22.1]{ShaoGao26}. In their notation, let
\[
 R_3=(S(abc)\times D_2)/(\{\infty\}\times D_2),
 \qquad \mathcal R_3=\mathsf V(R_3).
\]
Their results give
\begin{equation}\label{eq:shao-comparison}
 \mathcal A\subsetneq\mathcal M\subseteq\mathcal R_3,
 \qquad \mathcal M\cap L_{\mathrm{SG}}=\mathcal A,
 \qquad L_{\mathrm{SG}}\cap\operatorname{Mod}(\rho_0)=\mathcal A,
\end{equation}
where $\rho_0$ is the rectangular absorption identity
\begin{equation}\label{eq:rectangular}
 \rho_0:\quad xy+xz+tz\approx xy+xz+tz+ty.
\end{equation}
Here $\operatorname{Mod}(\rho_0)$ denotes the variety of ai-semirings
satisfying $\rho_0$. The last equality in \eqref{eq:shao-comparison}
is~\cite[Corollary~7.12]{ShaoGao26}.

\begin{proposition}\label{prop:exact-intersection}
The limit variety $\mathcal U=\mathsf V(T_\infty)$ satisfies
\[
 \mathcal U\cap L_{\mathrm{SG}}=\mathcal A.
\]
Consequently, $\mathcal U$ and $L_{\mathrm{SG}}$ are distinct and
incomparable limit varieties.
\end{proposition}
\begin{proof}
We first establish $\mathcal A\subseteq\mathcal U$ by
constructing a concrete homomorphic image of $T_2$. In $T_2$ put
\[
 I=\{\infty,q_1,p_2\}.
\]
By the multiplication rules in
Subsection~\ref{sec:path-construct}, each of $q_1$ and
$p_2$ annihilates every element on either side, with product $\infty$.
Since the addition is flat, $I+T_2\subseteq I$ and
$IT_2\cup T_2I\subseteq I$. Thus $I$ is a semiring ideal.
The quotient $T_2/I$ has the seven elements represented by
$\infty,u_1,e,u_2,p_1,q_2,\omega$. The assignment
\[
 [I]\mapsto\infty,\quad
 u_1\mapsto a,\quad e\mapsto b,\quad u_2\mapsto c,
 \quad p_1\mapsto ab,\quad q_2\mapsto bc,\quad\omega\mapsto abc
\]
is a semiring isomorphism $T_2/I\cong S(abc)$: its only
nonabsorbing products of distinct positive factors are
$u_1e=p_1$, $eu_2=q_2$, and
$p_1u_2=u_1q_2=\omega$, precisely corresponding to
$ab$, $bc$, and $abc$. As $T_2$ is a subsemiring of $T_\infty$,
we obtain $\mathcal A\subseteq\mathcal U$.

\medskip\noindent Next we prove $T_\infty\models\rho_0$. Every multiplication row of
$T_\infty$ has pairwise distinct values other than $\infty$,
as observed in the proof of Lemma~\ref{lem:path-flat}. Under an arbitrary evaluation,
if $xy+xz+tz=\infty$, adjoining $ty$ leaves the value unchanged.
Otherwise flatness forces
\[
 xy=xz=tz\ne\infty.
\]
Left cancellation of nonabsorbing products in the row indexed by
$x$ implies $y=z$; consequently $ty=tz$ and the right-hand side
has the same value. Hence
$\mathcal U\subseteq\operatorname{Mod}(\rho_0)$.

\medskip\noindent The Shao--Gao intersection theorem now gives the exact
intersection.
Equation~\eqref{eq:shao-comparison} gives
\[
 \mathcal A\subseteq\mathcal U\cap L_{\mathrm{SG}}
 \subseteq L_{\mathrm{SG}}\cap\operatorname{Mod}(\rho_0)
 =\mathcal A,
\]
which establishes the equality. Since $\mathcal A$ is properly
contained in $L_{\mathrm{SG}}$ and $\mathcal U$ is a limit variety
by Theorem~\ref{thm:path-limit}, the two limit varieties are different. Two distinct
limit varieties cannot be comparable: otherwise the larger would
have a proper nonfinitely based subvariety. Therefore
\[
 \mathcal U\not\subseteq L_{\mathrm{SG}},
 \qquad L_{\mathrm{SG}}\not\subseteq\mathcal U.
\]
\end{proof}

\begin{proposition}\label{prop:other-limit}
With $\mathcal M$ as above, one has
\[
 \mathcal U\not\subseteq\mathcal R_3.
\]
In particular, $\mathcal U\ne\mathcal M$, and these two limit
varieties are incomparable:
\[
 \mathcal U\not\subseteq\mathcal M,
 \qquad\mathcal M\not\subseteq\mathcal U.
\]
\end{proposition}
\begin{proof}
Consider the identity
\begin{equation}\label{eq:tau-separation}
 \tau:\qquad xy+yz\approx x^2.
\end{equation}
It holds in $S(abc)$. Indeed, every square evaluates to $\infty$,
since the defining word has no repeated letter. If
$xy+yz\ne\infty$, flatness would give $xy=yz\ne\infty$.
The nonempty factor representing $y$ would then be a proper suffix
of the factor representing $xy$ and a proper prefix of that same
factor representing $yz$. This is impossible, because every
nonempty factor of $abc$ has distinct letters and hence no nonempty
proper border. Consequently both sides of $\tau$ always equal
$\infty$ in $S(abc)$.

For every evaluation in $S(abc)\times D_2$, the first
coordinates of both sides of $\tau$ therefore equal
$\infty$. Since the quotient defining $R_3$ collapses
all of $\{\infty\}\times D_2$, it follows that
$R_3\models\tau$, and so $\mathcal R_3\models\tau$.

\medskip\noindent By contrast, in $T_3$, take $x=p_1$, $y=u_2$, and $z=q_3$.
The multiplication rules give
\[
 xy=p_1u_2=\omega,
 \qquad yz=u_2q_3=\omega,
 \qquad x^2=p_1^2=\infty.
\]
Thus $xy+yz=\omega\ne\infty=x^2$, so
$T_3\not\models\tau$. Since $T_3\le T_\infty$,
$\mathcal U\not\subseteq\mathcal R_3$.
On the other hand $\mathcal M\subseteq\mathcal R_3$ by
\cite[Theorem~22.1]{ShaoGao26}. Therefore
$\mathcal U\not\subseteq\mathcal M$ and in particular
$\mathcal U\ne\mathcal M$.
Finally $\mathcal M$ and $\mathcal U$ are both limit varieties
by~\cite[Theorem~22.1]{ShaoGao26} and Theorem~\ref{thm:path-limit}, respectively.
The same minimality argument as in Proposition~\ref{prop:exact-intersection}
excludes the reverse containment.
\end{proof}

\begin{corollary}\label{cor:five-limits}
For every $n\ge2$, the variety $\W_n$ contains at least five
pairwise incomparable limit subvarieties:
\[
 \Var(SR_6),\qquad L_{\mathrm{SG}},\qquad
 \mathcal M,\qquad\mathcal U,\qquad\mathcal V_{\mathrm{ac}},
\]
where $\mathcal V_{\mathrm{ac}}$ is the acyclic directed-graph
limit variety of Gao--Ren as recorded in
Shao--Gao~\cite[Proposition~25.3]{ShaoGao26}.
\end{corollary}
\begin{proof}
Proposition~\ref{lem:sr6-hsp} gives $\Var(SR_6)\subseteq\W_1$.
By Shao--Gao~\cite[Theorem~1.2]{ShaoGao26},
$L_{\mathrm{SG}}\subseteq\Var(S(abc)^0)$, and
Lemma~\ref{lem:SG-cyclic-endpoints} puts this linear-word
variety in $\W_2$. Shao--Gao's Theorem~22.2 therefore gives
$\mathcal M\subseteq\W_2$; Theorem~\ref{thm:path-limit}
gives $\mathcal U\subseteq\W_2$.

Both $S(a^2)^0$ and $S(ab)^0$ are homomorphic images of
$S(abba)^0$. To see this, fix $\mathbf{v}\in\{bb,ab\}$ and take
the natural factor map $S(abba)\twoheadrightarrow S(\mathbf{v})$
which fixes the nonempty factors of $\mathbf{v}$ and sends all other
old factors, including the old top, to $\infty$.
This is a homomorphism because every factor of a retained
word is again retained; a discarded factor cannot occur in a
product whose result is retained. Extend this map to the zero
extensions by sending the \emph{new} element $0$ to the new
$0$. This proves the asserted surjections (renaming $b$
to $a$ in the first target). In particular, the map keeps the old top $\infty$ and the
new zero $0$ distinct.
Shao--Gao's Corollary~4.2 gives
\[
 \Var(S(a^2,ab)^0)=\Var(S(a^2)^0)\vee\Var(S(ab)^0)
 \subseteq\W_2.
\]
Their Proposition~25.3 now places $\mathcal V_{\mathrm{ac}}$
in $\W_2$. All five varieties thus belong to $\W_n$ for $n\ge2$.

The variety $\Var(SR_6)$ is commutative, whereas each of
the other four fails commutativity. This separates the first
variety from the other four. Moreover,
$\mathcal V_{\mathrm{ac}}\models xyz\approx x^2$,
while each of $L_{\mathrm{SG}},\mathcal M,\mathcal U$
contains $\Var(S(abc))$, which refutes that identity.
Theorem~22.1 of Shao--Gao separates $L_{\mathrm{SG}}$
from $\mathcal M$, and Propositions~\ref{prop:exact-intersection}
and~\ref{prop:other-limit} distinguish $\mathcal U$ from both.
Hence the five limit varieties are distinct. Two distinct
limit varieties cannot contain one another.
\end{proof}

The preceding propositions distinguish $\mathcal U$ from
both limit varieties of Shao--Gao. We make no claim about
the exact intersection $\mathcal U\cap\mathcal M$.

\section{Conclusion}\label{sec:conclusion}
The mirror-word zero extensions give the strict chain
$\W_1\subsetneq\W_2\subsetneq\cdots$.
By Theorem~\ref{thm:adjacent-cardinality}, every consecutive
interval $[\W_n,\W_{n+1}]$ ($n\ge1$) contains an
order-embedded power-set lattice and hence has cardinality
$2^{\aleph_0}$. For $n\ge2$ the same cardinality holds
for $[\V_n,\W_n]$ by
Corollary~\ref{cor:zero-extension-continuum}; this latter
argument uses the cyclic semirings of Shao--Gao and verifies
the endpoint conditions for the mirror words.
It follows that $|[\V_n,\W_\infty]|=2^{\aleph_0}$
for $n\ge2$. Proposition~\ref{prop:infinite-local}
establishes local finiteness of $\W_\infty$.

Finite local quotients, six-element evaluation kernels and
the uniform short-word bound establish infinite axiomatic rank
throughout $[\W_2,\W_\infty]$; see
Corollary~\ref{ind:cor:complexity} for variable bounds.
The finite-basis classification of individual zero extensions
instead follows from Shao--Gao~\cite{ShaoGao26}.

The variety $\Var(SR_6)$ is the unique limit subvariety of
$\W_1$ by Theorem~\ref{thm:unique-limit}. The directed-path
construction of Theorem~\ref{thm:path-limit} yields a distinct
limit variety $\mathcal U=\Var(T_\infty)$ that is
contained in every $\W_n$ ($n\ge2$) but is not contained in any
$\V_n$. Proposition~\ref{prop:exact-intersection} proves
\[
 \mathcal U\cap\Var(B_{15})=\Var(S(abc)).
\]
Together with the two limit varieties of Shao--Gao and
the previously known acyclic-graph example, this gives at
least five pairwise incomparable limit subvarieties of each
$\W_n$ for $n\ge2$, as recorded in
Corollary~\ref{cor:five-limits}.

\end{document}